\documentclass[10pt,a4paper]{article}
\usepackage[utf8]{inputenc}
\usepackage[T1]{fontenc}
\usepackage[margin=3cm]{geometry}
\usepackage{charter}
\usepackage[english]{babel}
\def\doi#1{\href{https://doi.org/\detokenize{#1}}{\url{https://doi.org/\detokenize{#1}}}}
\usepackage{amsmath}
\usepackage{amsfonts}
\usepackage{amssymb}
\usepackage{amsthm}
\usepackage{mathtools}
\usepackage{nicefrac}
\usepackage{graphicx}
\usepackage{xcolor}
\usepackage{dsfont}
\usepackage{todonotes}
\usepackage[affil-sl]{authblk}
\definecolor{darkred}{RGB}{150, 0, 0}
\definecolor{darkgreen}{RGB}{0, 150, 0}
\definecolor{darkblue}{RGB}{0, 0, 150}
\usepackage[breaklinks=true,colorlinks,citecolor=darkgreen,linkcolor=darkred,urlcolor=darkblue,bookmarks=false]{hyperref}

\usepackage{enumerate}
\usepackage{tikz}
\usepackage{booktabs}
\usepackage[inline]{enumitem}
\usepackage{xspace}
\usepackage{subcaption}
\usepackage[binary-units=true]{siunitx}
\usepackage{cite}

\usepackage{graphicx}

\usepackage{array}
\newcolumntype{L}[1]{>{\raggedright\let\newline\\\arraybackslash\hspace{0pt}}m{#1}}
\newcolumntype{C}[1]{>{\centering\let\newline\\\arraybackslash\hspace{0pt}}m{#1}}
\newcolumntype{R}[1]{>{\raggedleft\let\newline\\\arraybackslash\hspace{0pt}}m{#1}}

\usepackage{listings}
\newcommand{\R}{\mathds{R}}
\newcommand{\Z}{\mathds Z}

\newcommand{\feasregion}{\ensuremath{\mathrm\Phi}}

\newcommand{\symmetricgroup}[1]{\ensuremath{\mathcal{S}_{#1}}}
\newcommand{\define}{\ensuremath{\coloneqq}}
\newcommand{\scip}{{\tt SCIP}\xspace}
\newcommand{\etal}{et~al.\@\xspace}
\newcommand{\binary}{\ensuremath{\{0, 1\}}}
\newcommand{\symretope}{\ensuremath{\mathcal{X}}\xspace}
\newcommand{\coNP}{\ensuremath{\mathrm{coNP}}}
\newcommand{\NP}{\ensuremath{\mathrm{NP}}}
\newcommand{\bnbtree}{\ensuremath{\mathcal B}}
\newcommand{\bnbnodes}{\ensuremath{\mathcal V}}
\newcommand{\bnbedges}{\ensuremath{\mathcal E}}
\newcommand{\stab}[2]{\mathrm{stab}(#1, #2)}

\newcommand{\id}{\ensuremath{\mathrm{id}}}
\newcommand{\ihat}{\ensuremath{\hat\imath}}
\newcommand{\bigo}{\ensuremath{\mathop{\mathcal{O}}}}
\newcommand{\code}[1]{\texttt{#1}}
\newcommand{\eqreffromto}[2]{(\ref{#1}\nobreakdash--\ref{#2})}
\newcommand{\domain}{\ensuremath{\mathcal{D}}}

\newcommand{\optval}{\mathrm{OPT}}

\newcommand{\isomorphismpruning}{IsoPr\xspace}
\newcommand{\orbitalfixing}{OF\xspace}
\newcommand{\SHC}{SHC\xspace}
\newcommand{\SHCs}{\SHC{}s\xspace}
\newcommand{\VDR}{VDR\xspace}
\newcommand{\VDRs}{\VDR{}s\xspace}
\newcommand{\lexfix}{LexFix\xspace}
\newcommand{\lexred}{LexRed\xspace}
\newcommand{\bb}{B\&{}B\xspace}
\newcommand{\bnb}{\bb}

\newcommand{\card}[1]{|#1|}
\newcommand{\restrict}[2]{
	\ensuremath{
		#1
		\raisebox{-.1ex}{\ensuremath{\big\rvert}}{}_{\smash{#2}}
	}}

\newcommand{\blank}{\ensuremath{\_}}
\newcommand{\red}[1]{\color{red}{#1}}

\DeclareMathOperator{\lexmin}{lexmin}
\DeclareMathOperator{\lexmax}{lexmax}

\theoremstyle{plain}
\newtheorem{theorem}{Theorem}
\newtheorem{lemma}[theorem]{Lemma}
\theoremstyle{definition}
\newtheorem{remark}[theorem]{Remark}

\newtheorem{example}[theorem]{Example}
\newtheorem{proposition}[theorem]{Proposition}
\newtheorem{problem}[theorem]{Problem}

\tikzset{bbnode/.append style={draw=black,circle,inner sep=0mm,minimum 
size=5mm}}
\usetikzlibrary{trees}

\title{A Unified Framework for Symmetry Handling}

\author{Jasper van Doornmalen}
\author{Christopher Hojny}
\affil{Combinatorial Optimization Group, Eindhoven University of 
Technology\\
  email: \{m.j.v.doornmalen, c.hojny\}@tue.nl}
\date{}

\begin{document}

\maketitle              

\begin{abstract}
  Handling symmetries in optimization problems is essential
  for devising efficient solution methods.
  In this article, we present a general framework that captures many of the
  already existing symmetry handling methods.
  While these methods are mostly discussed independently from each other, our
  framework allows to apply different methods simultaneously and thus
  outperforming their individual effect.
  Moreover, most existing symmetry handling methods only apply to binary variables.
  Our framework allows to easily generalize these methods to
  general variable types.
  Numerical experiments confirm that our novel framework is superior to the
  state-of-the-art symmetry handling methods as implemented in the solver
  \scip on a broad set of instances.
  \smallskip

\end{abstract}

\section{Introduction}
We consider optimization problems~$\optval(f,\feasregion) \define 
\min \{ f(x) : x \in \feasregion\}$ for a real-valued function~$f\colon\R^n
\to \R$ and feasible region~$\feasregion \subseteq
\R^n$ such that~$\optval(f,\feasregion)$ can be solved by (spatial)
branch-and-bound (\bb)~\cite{LandDoig1960,HorstTuy1996}.
This class of problems is very rich and captures problems such as
mixed-integer linear programs and mixed-integer nonlinear programs.
The core of \bb-methods is to repeatedly partition the feasible
region~$\feasregion$ into smaller subregions~$\feasregion'$
and to solve the reduced problem~$\optval(f, \feasregion')$.
Subregions do not need to be explored further if
it is known that they do not contain optimal or improving solutions (i.e., pruning by bound),
or if the region becomes empty (i.e., pruning by feasibility).
The sketched mechanism allows to routinely solve problems with thousands of
variables and constraints.
If symmetries are present, however, plain \bb usually struggles with
solving optimization problems as we explain next.

A \emph{symmetry} of the optimization problem is a
bijection~$\gamma\colon\R^n \to \R^n$ that maps
solution vectors $x \in \R^n$ to solution vectors $\gamma(x)$ 
while preserving the objective value and feasibility state, i.e., $f(x) =
f(\gamma(x))$ holds for all~$x \in \R^n$ and $x \in \feasregion$ if and
only if $\gamma(x) \in \feasregion$.
Resulting from this definition,
the set of all symmetries of an optimization problem
forms a group~$\bar{\Gamma}$.
When enumerating the subregions~$\feasregion'$ in \bb, it might happen that several subregions at different parts of the branch-and-bound tree contain equivalent copies of (optimal) solutions.
This results in unnecessarily large \bb trees.
By exploiting the presence of symmetries,
one could enhance \bb by finding more reductions and pruning rules
that further restrict the (sub-)regions without sacrificing finding
optimal solutions to the original problem
\cite{KouyialisEtAl2019,Liberti2012,pfetsch2019computational}.

To handle symmetries, different approaches have been discussed in the
literature.
Two very popular classes of symmetry handling methods are
\emph{symmetry handling constraints}
(\SHCs)~\cite{BendottiEtAl2021,Friedman2007,Hojny2020,hojny2019polytopes,KaibelEtAl2011,KaibelPfetsch2008,Liberti2008,Liberti2012,Liberti2012a,LibertiOstrowski2014}
and
\emph{variable domain reductions} (\VDRs) derived from the
\bb-tree~\cite{Margot2003,ostrowski2009symmetry,OstrowskiEtAl2011}.
Both approaches remove symmetric solutions from the search space without
eliminating all optimal solutions.
As we detail in Section~\ref{sec:overview}, \SHCs and \VDRs come with a 
different flavor.
\SHCs usually use a static scheme for symmetry reductions, 
whereas \VDRs dynamically find reductions based on decisions in the \bb-tree.
In their textbook form, \SHCs and \VDRs are thus incompatible, i.e., they
cannot be combined and that might leave some potential for symmetry reductions
unexploited.
Moreover, many \SHCs and \VDRs have only been studied for binary variables
and for symmetries corresponding to permutations of variables,
which restricts their applicability.

The goal of this article is to overcome these drawbacks.
We therefore devise a unified framework for symmetry handling.
The contributions of our framework are that it
\begin{enumerate}[label={{(C\arabic*)}},ref={C\arabic*}]
\item\label{contrib1} resolves incompatibilities between \SHCs and \VDRs,
\item\label{contrib2} applies for general variable types, and
\item\label{contrib3} can handle symmetries of arbitrary finite groups,
which are not necessarily permutation groups.
\end{enumerate}
Due to~\ref{contrib1}, one is thus not restricted anymore to either use
\SHCs or \VDR methods.
In particular, we show that many popular \VDR techniques for binary
variables such as orbital fixing~\cite{OstrowskiEtAl2011} and isomorphism
pruning~\cite{Margot2002,ostrowski2009symmetry}, but also \SHCs can be
simultaneously cast into our framework.
That is, our framework unifies the application of these techniques.
To fully facilitate our framework regarding~\ref{contrib2}, the second
contribution of this paper is a generalization of many symmetry handling
techniques from binary variables to general variable types.
This allows for handling symmetries in more classes of optimization
problems, in particular classes with non-binary variables.

Regarding~\ref{contrib3}, we stress that this result is not based on the
observation that every finite group is isomorphic to a permutation
group by Cayley's theorem~\cite{AlperinBell1995}, because the space in
which the isomorphic permutation group is acting might differ from~$\R^n$.

\paragraph{Outline}
After providing basic notations and definitions, Section~\ref{sec:overview}
provides an overview of existing symmetry handling methods.
In particular, we illustrate the techniques that we will later on cast into
our unified framework.
The framework itself will be introduced in Section~\ref{sec:framework}.
Section~\ref{sec:cast} shows how existing symmetry handling methods can be
used in our framework and how these methods can be generalized from binary
to general variables.
We conclude this article in Section~\ref{sec:num} with an extensive
numerical study of our new framework both for specific applications and
benchmarking instances.
The study reveals that our novel framework is substantially faster than the
state-of-the-art methods on both \SHCs and VDRs as implemented in the
solver \scip.

\paragraph{Notation and Definitions}
Throughout the article, we assume that we have access to a group~$\Gamma$
consisting of (not necessarily all) symmetries 
of the optimization problem $\optval(f, \feasregion)$.
That is, $\Gamma$ is a subgroup of $\bar{\Gamma}$,
which we denote by~$\Gamma \leq \bar{\Gamma}$.
We refer to~$\Gamma$ as a symmetry group of the problem.
For solution vectors~$x \in \R^n$, 
the set of symmetrically equivalent solutions is
its~\emph{$\Gamma$-orbit}~$\{\gamma(x) : \gamma\in\Gamma\}$.

Let~$\symmetricgroup{n}$ be the \emph{symmetric group} of~$[n] \define
\{1,\dots,n\}$.
Moreover, let~$[n]_0 \define [n] \cup \{0\}$.
Being in line with the existing literature on symmetry handling,
we assume that permutations~$\gamma \in \symmetricgroup{n}$ act on
vectors~$x \in \R^n$ by permuting their index sets, i.e.,
$\gamma(x) \define ( x_{\gamma^{-1}(i)} )_{i=1}^n$.
We call such symmetries \emph{permutation symmetries}.
The identity permutation is denoted by~$\id$.
To represent permutations~$\gamma$, we use their disjoint cycle
representation, i.e., $\gamma$ is the composition of disjoint cycles $(i_1,
\dots, i_r)$ such that $\gamma(i_k) = i_{k + 1}$ for $k \in \{1,\dots,
r-1\}$ and $\gamma(i_r) = i_1$.

In practice, the symmetry group $\Gamma$ is either provided by a user
or found using detection methods such as in~\cite{pfetsch2019computational,salvagnin2005dominance}.
Detecting the full permutation symmetry group for binary problems, 
however, is \NP-hard~\cite{margot2010symmetry}.
For non-linear problems,
depending on how the feasible region~$\feasregion$ is given,
already verifying if~$\gamma$ is a symmetry might be 
undecidable~\cite{Liberti2012}.

To handle symmetries, among others, we will make use of variable domain
propagation.
The idea of propagation approaches is, given a symmetry reduction rule and
domains for all variables, to derive reductions of some variable domains if
every solution adhering to the symmetry rule is contained in the reduced
domain.
More concretely, let~$\feasregion'$ be the feasible region of some
subproblem encountered during branch-and-bound.
For every variable $x_i$, $i \in [n]$, let~$\domain_i \subseteq \R$ be its
domain, which covers the projection of $\feasregion'$ on $x_i$, 
i.e., $\domain_i \supseteq \{ v \in \R : 
x_i = v \text{ for some } x \in \feasregion'\}$.
In an integer programming context, the domain $\domain_i$ corresponds to an interval in practice.
A symmetry reduction rule is encoded as a set~$\mathcal{C} \subseteq \R^n$,
which consists of all solution vectors that adhere to the rule.
The goal of \emph{variable domain propagation} is to find 
sets~$\domain'_i \subseteq \domain_i$, $i \in [n]$, such that
$
\mathcal{C} \cap \bigtimes_{i = 1}^n \domain'_i
=
\mathcal{C} \cap \bigtimes_{i = 1}^n \domain_i
$.
In this case, the domain of variable~$x_i$ can be reduced to~$\domain'_i$.
We say that propagation is \emph{complete} if, for every~$i \in [n]$,
domain~$\domain'_i$ is inclusionwise minimal.

Throughout this article, we denote \emph{full} \bb-trees by~$\bnbtree =
(\bnbnodes, \bnbedges)$, i.e., we do not prune nodes by their objective
value and do not apply enhancements such as cutting planes or bound
propagation.
This is only required to prove \emph{theoretical} statements about symmetry
handling and does not restrict their practical applicability as we will
discuss below.
If not mentioned differently, we assume~$\bnbtree$ to be finite, which
might not be the case for spatial branch-and-bound; the case of infinite
\bnb-trees will be discussed separately.
For~$\beta \in \bnbnodes$, let~$\chi_\beta$ be the set of its children
and let~$\feasregion(\beta) \subseteq \feasregion$ be the feasible 
solutions at $\beta$, i.e., the intersection of $\feasregion$
and the branching decisions.
If~$\beta$ is not a leaf, we assume that~$\feasregion(\omega)$, $\omega \in
\chi_\beta$, partitions~$\feasregion(\beta)$.
In our definitions, this is even the case for spatial branch-and-bound,
meaning that the partitioned feasible regions are not necessarily
closed sets.
We will discuss the practical consequences of this assumption below.

\section{Overview of symmetry handling methods for binary programs}
\label{sec:overview}

This section provides an overview of symmetry handling methods.
The methods lexicographic fixing, orbitopal fixing, isomorphism pruning,
and orbital fixing are described in detail, because we will later on show
that these methods can be cast into our framework and can be generalized
from binary to arbitrary variable domains.
Further symmetry handling methods will only be mentioned briefly.
We illustrate the different methods using the following running example.

\begin{problem}[NDB]
\label{prob:NDbinary}
Sherali and Smith~\cite{sherali2001models}
consider the \emph{Noise Dosage Problem} (ND).
There are $p$ machines, and on every machine a number of tasks
must be executed. For machine $i \in [p]$,
there are $d_i$ work cycles, each requiring $t_i$ hours of operation,
and each such work cycle induces $\alpha_i$ units of noise.
There are $q$ workers to be assigned to the machines, 
each of which is limited to $H$ hours of work.
The problem is to minimize the noise dosage of the worker
that receives the most units of noise.
We extend this problem definition with the requirement
that each worker can only be assigned once to the same machine,
which makes the problem a binary problem (NDB), namely to
\begin{subequations}
\makeatletter
\def\@currentlabel{NDB}
\makeatother
\renewcommand{\theequation}{NDB\arabic{equation}}%
\label{eq:NDB}
\begin{align}
\text{minimize}\
\eta &, \\
\text{subject to}\ 
\eta &\geq \sum\nolimits_{i \in [p]} \alpha_i \vartheta_{i,j} 
&& \text{for all}\ j \in [q],
\\
\sum\nolimits_{j \in [n]} \vartheta_{i,j} &= d_i
&& \text{for all}\ i \in [p],
\label{eq:ndb:demand}
\\
\sum\nolimits_{i \in [m]} t_i \vartheta_{i, j} &\leq H
&& \text{for all}\ j \in [q],
\label{eq:ndb:time}
\\
\vartheta &\in \binary^{p \times q},
\label{prob:ND:binary}
\\
\eta &\geq 0.
\end{align}
\end{subequations}
For a solution, $\vartheta$ represent
the worker schedules in a~$p \times q$ binary matrix.
The value of variable $\vartheta_{i, j}$ states how many tasks
on machine $i$ are allocated to worker $j$.
Since all workers have the same properties in this model,
symmetrically equivalent solutions are found by permuting 
the worker schedules.
This corresponds to permuting the columns of the $\vartheta$-matrix.
As such, a symmetry group of this problem is the group $\Gamma$
consisting of all column permutations of this $p \times q$ matrix.
\end{problem}

For illustration purposes, we focus on an NDB instance with ${p=3}$ machines and
${q=5}$ workers.
We stress that the symmetry handling methods work even if
variable domain reductions inferred by the model constraints are applied.
For the ease of presentation, however, we assume no such reductions
are made in the NDB problem instance.
For this reason, we do not specify~$d_i$, $t_i$, and~$H$.

\subsection{Symmetry handling constraints based on lexicographic order}
\label{sec:shcbin}

The philosophy of \emph{symmetry handling constraints} (\SHCs) is to
restrict the feasible region of an optimization problem to representatives
of the~$\Gamma$-orbits of feasible solutions.
A common way to do this, is to enforce that feasible solutions must be
\emph{lexicographically maximal} in their
$\Gamma$-orbit~\cite{Friedman2007}.

Let $x, y \in \R^n$. We say $x$ is
\emph{lexicographically larger} than $y$, denoted $x \succ y$,
if for some $k \in [n]$ we have $x_i = y_i$ for $i < k$, and $x_k > y_k$.
If $x \succ y$ or $x = y$, we write $x \succeq y$.
Since the lexicographic order specifies a total ordering on~$\R^n$,
to solve the optimization problem 
$\optval(f, \feasregion)$,
it is sufficient to consider only those solutions $x$
that are lexicographically maximal in their $\Gamma$-orbit. Let
\[
  \symretope \define \{ x \in \binary^n : x \succeq \gamma(x)
  \text{ for all } \gamma \in \Gamma \}.
\]
Then, solving~$\optval(f, \feasregion \cap \symretope)$
yields the same optimal objective and the same feasibility state
as the original problem.
Note, however, that deciding whether a vector~$x \in \binary^n$
is contained in $\symretope$ is \coNP-complete~\cite{babai1983canonical}.
Complete propagation of the \SHCs~\symretope is thus \coNP-hard for
general groups.
In practice, one therefore either neglects the group structure or applies
specialized algorithms for particular
groups~\cite{BendottiEtAl2021,doornmalenhojny2022cyclicsymmetries}.
We discuss lexicographic fixing and orbitopal fixing as representatives for
these two approaches.

\subsubsection{Lexicographic fixing (\lexfix)}
\label{sec:lexfix}
Instead of handling $x \succeq \gamma(x)$ for all $\gamma \in \Gamma$,
one can handle this \SHC for a single permutation~$\gamma$ only.
For binary problems, \cite{friedman2007fundamental} shows that~$x
\succeq \gamma(x)$ is equivalent to the linear inequality~$\sum_{k=1}^n 2^{n-k} x_k 
\geq \sum_{k=1}^n 2^{n-k} \gamma(x)_k$.
Due to the large coefficients, however, these inequalities might cause
numerical instabilities.
To circumvent numerical instabilities, \cite{hojny2019polytopes} presents
an alternative family of linear inequalities modeling~$x \succeq \gamma(x)$
in which all variable coefficients are either~$0$ or~$\pm 1$ and that can
be separated efficiently.

Alternatively, $x \succeq \gamma(x)$ can also be enforced using a
complete propagation algorithm that runs in linear
time~\cite{BestuzhevaEtal2021OO,doornmalenhojny2022cyclicsymmetries}.
Since a variable domain reduction in the binary setting corresponds to
fixing a variable, we refer to this algorithm as the \emph{lexicographic
fixing algorithm}, or \lexfix in short.

Using our running example NDB, we illustrate the idea of \lexfix for the
permutation~$\gamma$ that exchanges column~$2$ and~$3$ and fixes all
remaining columns.
Since the lexicographic order depends on a specific variable ordering, we
assume that the variables of the~$\vartheta$-matrix are sorted row-wise.
That is, $x = (\vartheta_{1,1}, \dots, \vartheta_{1, 5};
\dots; \vartheta_{3,1}, \dots, \vartheta_{3, 5})$.
We omit $\eta$ from the vector, since the orbit of $\eta$
is trivial with respect to $\Gamma$.

When removing fixed points 
of the solution vector from~$\gamma$,
enforcement of~$x \succeq \gamma(x)$ corresponds to
$
(\vartheta_{1,2}, \vartheta_{1,3}, 
\vartheta_{2,2}, \vartheta_{2,3}, 
\vartheta_{3,2}, \vartheta_{3,3}) 
\succeq 
(\vartheta_{1,3}, \vartheta_{1,2}, 
\vartheta_{2,3}, \vartheta_{2,2}, 
\vartheta_{3,3}, \vartheta_{3,2}) 
$,
which in turn corresponds to
$
(\vartheta_{1,2},
\vartheta_{2,2},
\vartheta_{3,2})
\succeq 
(\vartheta_{1,3},
\vartheta_{2,3},
\vartheta_{3,3}) 
$.
Complete propagation of that constraint for the running example
is shown in Figure~\ref{fig:branching:lexfix}.
For instance, in the leftmost node, $\vartheta_{1,3}$ can be fixed to~0,
because any solution with~$\vartheta_{1,3} = 1$ and satisfying the
remaining local variable domains violates the lexicographic order
constraint as~$(\vartheta_{1,2},\vartheta_{2,2},\vartheta_{3,2}) =
(0,\vartheta_{2,2},\vartheta_{3,2}) \nsucceq (1,0,\vartheta_{3,3}) =
(\vartheta_{1,3},\vartheta_{2,3},\vartheta_{3,3})$.

\begin{figure}[!tbp]
\centering
\begin{tikzpicture}[
	x=1cm, y=1cm, level distance=1.5cm,
	level 1/.style={sibling distance=7cm},
	level 2/.style={sibling distance=3.5cm},
	level 3/.style={sibling distance=1.5cm},
	font=\footnotesize
]

\node {$\left[\begin{array}{*5{@{}wc{2mm}@{}}}
		\blank&\blank&\blank&\blank&\blank\\
		\blank&\blank&\blank&\blank&\blank\\
		\blank&\blank&\blank&\blank&\blank
	\end{array}\right]$}
	child {
		node {$\left[\begin{array}{*5{@{}wc{2mm}@{}}}
			\blank&\blank&\blank&\blank&\blank\\
			\blank&\blank&     0&\blank&\blank\\
			\blank&\blank&\blank&\blank&\blank
		\end{array}\right]$}
		child {
			node {$\left[\begin{array}{*5{@{}wc{2mm}@{}}}
				\blank&     0&\red 0&\blank&\blank\\
				\blank&\blank&     0&\blank&\blank\\
				\blank&\blank&\blank&\blank&\blank
			\end{array}\right]$}
			edge from parent node[left] {$\vartheta_{1, 2} \gets 0$};
		}
		child {
			node {$\left[\begin{array}{*5{@{}wc{2mm}@{}}}
				\blank&1&\blank&\blank&\blank\\
				\blank&\blank&0&\blank&\blank\\
				\blank&\blank&\blank&\blank&\blank
			\end{array}\right]$}
			child {
				node {$\left[\begin{array}{*5{@{}wc{2mm}@{}}}
					\blank&1&0&\blank&\blank\\
					\blank&\blank&0&\blank&\blank\\
					\blank&\blank&\blank&\blank&\blank
				\end{array}\right]$}
				edge from parent node[left] {$\vartheta_{1, 3} \gets 0$};
			}
			child {
				node {$\left[\begin{array}{*5{@{}wc{2mm}@{}}}
					\blank&1&1&\blank&\blank\\
					\blank&\blank&0&\blank&\blank\\
					\blank&\blank&\blank&\blank&\blank
				\end{array}\right]$}
				edge from parent node[right] {$\vartheta_{1, 3} \gets 1$};
			}
			edge from parent node[right] {$\vartheta_{1, 2} \gets 1$};
		}
		edge from parent node[left] {$\vartheta_{2, 3} \gets 0$};
	}
	child {
		node {$\left[\begin{array}{*5{@{}wc{2mm}@{}}}
			\blank&\blank&\blank&\blank&\blank\\
			\blank&\blank&     1&\blank&\blank\\
			\blank&\blank&\blank&\blank&\blank
		\end{array}\right]$}
		child{
			node {$\left[\begin{array}{*5{@{}wc{2mm}@{}}}
				\blank&     0&\red 0&\blank&\blank\\
				\blank&\red 1&     1&\blank&\blank\\
				\blank&\blank&\blank&\blank&\blank
			\end{array}\right]$}
			edge from parent node[left] {$\vartheta_{1, 2} \gets 0$};
		}
		child{
			node {$\left[\begin{array}{*5{@{}wc{2mm}@{}}}
				\blank&1&\blank&\blank&\blank\\
				\blank&\blank&1&\blank&\blank\\
				\blank&\blank&\blank&\blank&\blank
			\end{array}\right]$}
			child {
				node {$\left[\begin{array}{*5{@{}wc{2mm}@{}}}
					\blank&1&0&\blank&\blank\\
					\blank&\blank&1&\blank&\blank\\
					\blank&\blank&\blank&\blank&\blank
				\end{array}\right]$}
				edge from parent node[left] {$\vartheta_{1, 3} \gets 0$};
			}
			child {
				node {$\left[\begin{array}{*5{@{}wc{2mm}@{}}}
					\blank&     1&     1&\blank&\blank\\
					\blank&\red 1&     1&\blank&\blank\\
					\blank&\blank&\blank&\blank&\blank
				\end{array}\right]$}
				edge from parent node[right] {$\vartheta_{1, 3} \gets 1$};
			}
			edge from parent node[right] {$\vartheta_{1, 2} \gets 1$};
		}
		edge from parent node[right] {$\vartheta_{2, 3} \gets 1$};
	}
;

\end{tikzpicture}
\caption{Branch-and-bound tree for the NDB problem. Fixings by \lexfix are drawn red.}
\label{fig:branching:lexfix}
\end{figure}
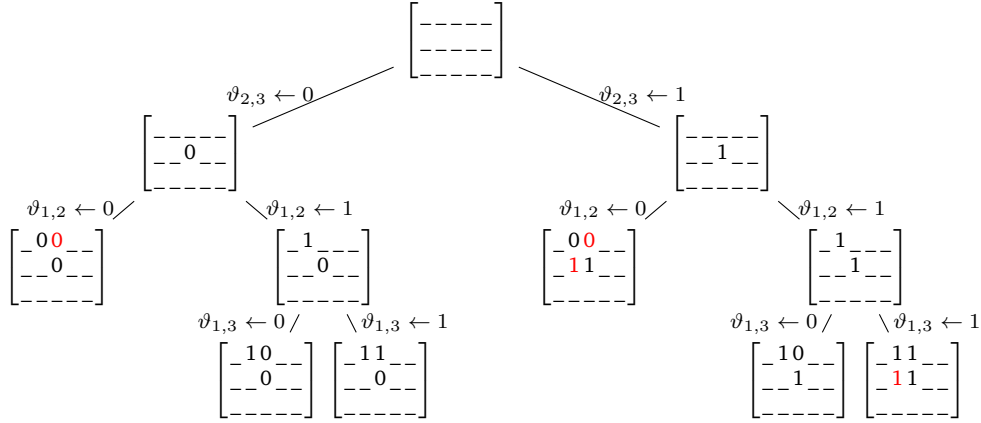

\subsubsection{Orbitopal fixing}
\label{sec:orbitopalfixing}

Note that \lexfix neglects the entire group structure and thus might not
find variable domain reductions that are based on the interplay of
different symmetries.
Since propagation for~$\symretope$ is \coNP-hard, special cases of
groups have been investigated that appear very frequently in
practice.
One of these groups corresponds to the symmetries present in NDB, i.e., the
symmetry group~$\Gamma$ acts on $p \times q$ matrices of binary variables
by exchanging their columns.
We refer to such matrix symmetries as \emph{orbitopal} symmetries.
Besides in NDB, orbitopal symmetries arise in many further applications
such as graph coloring or unit commitment
problems~\cite{BendottiEtAl2021,KaibelEtAl2011,margot2007symmetric}.

For the variable ordering discussed in Section~\ref{sec:lexfix}
and~$\Gamma$ being a group of orbitopal symmetries, one can show that
enforcing~$x \succeq \gamma(x)$ for all~$\gamma \in \Gamma$ is equivalent
to sorting the columns of the variable matrix in lexicographically
non-increasing order.
Bendotti et al.~\cite{BendottiEtAl2021} present a propagation
algorithm for such symmetries, so-called \emph{orbitopal fixing}.
Kaibel \etal~\cite{KaibelEtAl2011}
discuss a propagation algorithm for the case that each row of the variable
matrix has at most one~$1$-entry.
Both algorithms are complete and run in linear time.
Moreover, Kaibel and Pfetsch~\cite{KaibelPfetsch2008} derive a facet description of all
binary matrices with lexicographically sorted columns and at most (or
exactly) one~1-entry per row.
That is, the \SHC~$\symretope$ can be replaced by the facet description in
this case.

Given initial variable domains $\domain \subseteq \binary^{p \times
  q}$, the algorithm of Bendotti \etal finds the tightest variable domains
as follows.
First, the lexicographically minimal and maximal matrices in $\symretope
\cap \domain$ are computed.
Then, for each column in the variable matrix, the associated
variables can be fixed to the value of the lexicographically extreme matrices
up to the first row where these extremal matrices differ.
If the columns of the extremal matrices are identical, 
the whole column can be fixed.

For the running example, Figure~\ref{fig:branching:orbitopalfixing}
presents the branch-and-bound tree with variable fixings
by orbitopal fixing.
For instance, if~$(\vartheta_{2,3}, \vartheta_{1,2}) \gets (0, 0)$,
the lexicographically minimal and maximal matrices are
{\footnotesize $
\left[\begin{array}{*5{@{}wc{2mm}@{}}}
0&\underline 0&0&0&0\\
0&0&\underline 0&0&0\\
0&0&0&0&0
\end{array}\right]
$}
and
{\footnotesize $
\left[\begin{array}{*5{@{}wc{2mm}@{}}}
1&\underline 0&0&0&0\\
1&1&\underline 0&0&0\\
1&1&1&1&1
\end{array}\right]
$}, with the branching decisions underlined,
respectively.
Applying orbitopal fixing then leads to the leftmost matrix in
Figure~\ref{fig:branching:orbitopalfixing}.

\begin{figure}[!tbp]
\centering
\begin{tikzpicture}[
	x=1cm, y=1cm, level distance=1.5cm,
	level 1/.style={sibling distance=7cm},
	level 2/.style={sibling distance=3.5cm},
	level 3/.style={sibling distance=1.5cm},
	font=\footnotesize
]

\node {$\left[\begin{array}{*5{@{}wc{2mm}@{}}}
		\blank&\blank&\blank&\blank&\blank\\
		\blank&\blank&\blank&\blank&\blank\\
		\blank&\blank&\blank&\blank&\blank
	\end{array}\right]$}
	child {
		node {$\left[\begin{array}{*5{@{}wc{2mm}@{}}}
			\blank&\blank&\blank&\blank&\blank\\
			\blank&\blank&     0&\blank&\blank\\
			\blank&\blank&\blank&\blank&\blank
		\end{array}\right]$}
		child {
			node {$\left[\begin{array}{*5{@{}wc{2mm}@{}}}
				\blank&     0&\red 0&\red 0&\red 0\\
				\blank&\blank&     0&\red 0&\red 0\\
				\blank&\blank&\blank&\blank&\blank
			\end{array}\right]$}
			edge from parent node[left] {$\vartheta_{1, 2} \gets 0$};
		}
		child {
			node {$\left[\begin{array}{*5{@{}wc{2mm}@{}}}
				\red 1&1&\blank&\blank&\blank\\
				\blank&\blank&0&\blank&\blank\\
				\blank&\blank&\blank&\blank&\blank
			\end{array}\right]$}
			child {
				node {$\left[\begin{array}{*5{@{}wc{2mm}@{}}}
					1&1&0&\red 0&\red 0\\
					\blank&\blank&0&\red 0&\red 0\\
					\blank&\blank&\blank&\blank&\blank
				\end{array}\right]$}
				edge from parent node[left] {$\vartheta_{1, 3} \gets 0$};
			}
			child {
				node {$\left[\begin{array}{*5{@{}wc{2mm}@{}}}
					     1&     1&     1&\blank&\blank\\
					\blank&\blank&     0&\blank&\blank\\
					\blank&\blank&\blank&\blank&\blank
				\end{array}\right]$}
				edge from parent node[right] {$\vartheta_{1, 3} \gets 1$};
			}
			edge from parent node[right] {$\vartheta_{1, 2} \gets 1$};
		}
		edge from parent node[left] {$\vartheta_{2, 3} \gets 0$};
	}
	child {
		node {$\left[\begin{array}{*5{@{}wc{2mm}@{}}}
			\blank&\blank&\blank&\blank&\blank\\
			\blank&\blank&     1&\blank&\blank\\
			\blank&\blank&\blank&\blank&\blank
		\end{array}\right]$}
		child{
			node {$\left[\begin{array}{*5{@{}wc{2mm}@{}}}
				\blank&     0&\red 0&\red 0&\red 0\\
				\blank&\red 1&     1&\blank&\blank\\
				\blank&\blank&\blank&\blank&\blank
			\end{array}\right]$}
			edge from parent node[left] {$\vartheta_{1, 2} \gets 0$};
		}
		child{
			node {$\left[\begin{array}{*5{@{}wc{2mm}@{}}}
				\red 1&     1&\blank&\blank&\blank\\
				\blank&\blank&     1&\blank&\blank\\
				\blank&\blank&\blank&\blank&\blank
			\end{array}\right]$}
			child {
				node {$\left[\begin{array}{*5{@{}wc{2mm}@{}}}
					     1&     1&     0&\red 0&\red 0\\
					\blank&\blank&     1&\blank&\blank\\
					\blank&\blank&\blank&\blank&\blank
				\end{array}\right]$}
				edge from parent node[left] {$\vartheta_{1, 3} \gets 0$};
			}
			child {
				node {$\left[\begin{array}{*5{@{}wc{2mm}@{}}}
					     1&     1&     1&\blank&\blank\\
					\red 1&\red 1&     1&\blank&\blank\\
					\blank&\blank&\blank&\blank&\blank
				\end{array}\right]$}
				edge from parent node[right] {$\vartheta_{1, 3} \gets 1$};
			}
			edge from parent node[right] {$\vartheta_{1, 2} \gets 1$};
		}
		edge from parent node[right] {$\vartheta_{2, 3} \gets 1$};
	}
;

\end{tikzpicture}
\caption{Branch-and-bound tree for the NDB problem.
Fixings by orbitopal fixing are drawn red.}
\label{fig:branching:orbitopalfixing}
\end{figure}
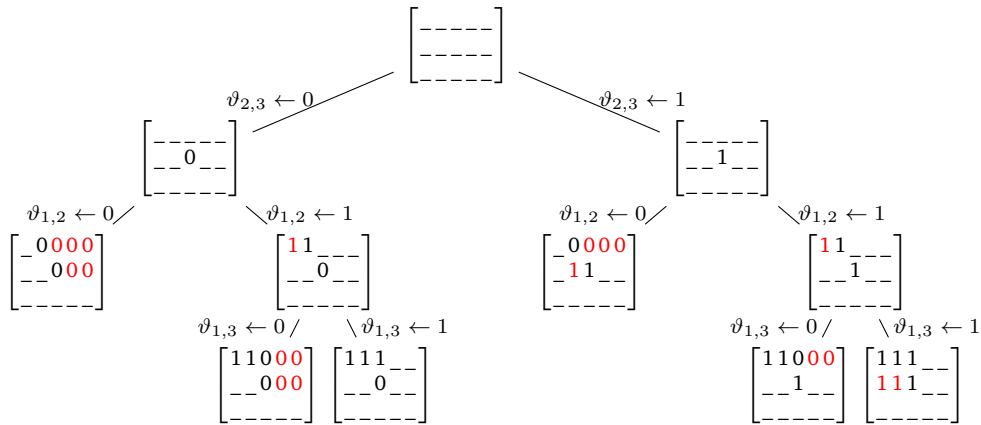

Note that orbitopal fixing does not find any variable domain reduction
after the first branching decision in
Figure~\ref{fig:branching:orbitopalfixing}.
The reason is that branching occurred for a variable in the second row.
To still be able to benefit from some symmetry reductions in this case,
Bendotti \etal~\cite{BendottiEtAl2021} also discuss a variant of orbitopal
fixing that adapts the order of rows based on the branching decisions.
They empirically show that this adapted algorithm performs better than the
original algorithm for the unit commitment problem.
We discuss the adapted variant in more detail and with more flexibility in
terms of our new framework in Section~\ref{sec:framework} in
Example~\ref{ex:orbitopalfixing}.

\subsection{Symmetry reductions based on branching tree structure}
\label{sec:overviewtree}

Recall that the \SHCs discussed in the previous section restrict the
feasible region of an optimization problem.
That is, already before solving the optimization problem, it is determined
which symmetric solutions are discarded.
A second family of symmetry handling techniques uses a more dynamic
approach, which prevents to create symmetric copies of subproblems already
created in the branch-and-bound tree.
The motivation of this is that symmetry reductions can be carried out earlier than in a static setting as described in the previous section.

Throughout this section, let~$\bnbtree = (\bnbnodes, \bnbedges)$ be a
branch-and-bound tree, where branching is applied on a single variable.
For a node $\beta \in \bnbnodes$, let $B_0^\beta$ (resp.~$B_1^\beta$) 
be the set of variable indices of a solution vector
that are fixed to~0 (resp.~1) by the branching decisions on
the rooted tree path to $\beta$.
Moreover, we assume~$\feasregion \subseteq \binary^n$ as the techniques
that we describe next have mostly been discussed for binary problems.

\subsubsection{Isomorphism pruning}
\label{sec:isomorphismpruning}

In classical branch-and-bound approaches, a node can be pruned if the
corresponding subproblem is infeasible (pruning by infeasibility) or if the
subproblem cannot contain an improving solution (pruning by bound).
In the presence of symmetry, Margot~\cite{Margot2002,Margot2003} and
Ostrowski~\cite{ostrowski2009symmetry} discuss another pruning rule that
discards symmetric or isomorphic subproblems, so-called \emph{isomorphism
  pruning}.

The least restrictive version of isomorphism pruning is due to
Ostrowski~\cite{ostrowski2009symmetry}.
Note that the way how we phrase isomorphism pruning differs from the
notation in~\cite{ostrowski2009symmetry}.
In terms of our unified framework that we discuss in
Section~\ref{sec:framework}, however, our notation is more suitable.

Let~$\beta \in \bnbnodes$ be a branch-and-bound tree node at depth~$m$,
and suppose that every branching decision corresponds to a single variable fixing.
Let $i_k$ be the index of the branching variable at depth~$k \in [m]$ on
the rooted path to~$\beta$.
Then, $B_0^\beta \cup B_1^\beta = \{i_1, \dots, i_m\}$.
Let~$\pi_\beta \in \symmetricgroup{n}$ be any permutation
with $\pi_\beta(i_k) = k$ for $k \in [m]$ and let~$y \in \binary^n$ such
that~$y_i = 1$ if and only if~$i \in B_1^\beta$.
For a vector~$x \in \R^n$ and~$A \subseteq [n]$, we denote
by~$\restrict{x}{A}$ its restriction to the entries in~$A$.
\begin{theorem}[Isomorphism Pruning]
  \label{thm:isoprune}
  Let~$\beta \in \bnbnodes$ be a node at depth~$m$.
  Node~$\beta$ can be pruned if there exists~$\gamma \in \Gamma$ such that
  $\restrict{\pi_\beta(y)}{[m]} \prec \restrict{\pi_\beta(\gamma(y))}{[m]}$.
\end{theorem}
Testing if a vector is lexicographically maximal in its orbit
is a \coNP-complete problem~\cite{babai1983canonical}.
As such, deciding if $\beta$ can be pruned by isomorphism
is an \NP-complete problem.

\begin{remark}
  Margot~\cite{margot2007symmetric} also describes a variant of isomorphism pruning that can be used to
  handle symmetries of general integer variables.
  Margot's variant assumes a specific branching rule.
  We do not describe it in more detail as our framework can also
  handle general integer variables while not relying on any assumptions on the 
  branching rule such as Ostrowski's version for binary variables.
\end{remark}

Figure \ref{fig:branching:isomorphismpruning} shows the branch-and-bound
tree after applying isomorphism pruning.
The only node~$\beta$ that can be pruned is where
$(\vartheta_{2,3},\vartheta_{1,2},\vartheta_{1,3}) = (0, 0, 1)$.
This is due to the symmetry $\gamma$ swapping column~$2$ and~$3$.
For this node, 
$\pi_\beta(x) = (\vartheta_{2,3},\vartheta_{1,2},\vartheta_{1,3})$
and
$(\pi_\beta \circ \gamma)(x) = 
(\vartheta_{2,2},\vartheta_{1,3},\vartheta_{1,2})$,
and as such 
$
\pi_\beta(y)
= (0, 0, 1, \dots) \prec (0, 1, 0, \dots)
= \pi_\beta(\gamma(y))
$.

\begin{figure}[!tbp]
\centering
\begin{tikzpicture}[
	x=1cm, y=1cm, level distance=1.5cm,
	level 1/.style={sibling distance=7cm},
	level 2/.style={sibling distance=3.5cm},
	level 3/.style={sibling distance=1.5cm},
	font=\footnotesize
]

\node {$\left[\begin{array}{*5{@{}wc{2mm}@{}}}
		\blank&\blank&\blank&\blank&\blank\\
		\blank&\blank&\blank&\blank&\blank\\
		\blank&\blank&\blank&\blank&\blank
	\end{array}\right]$}
	child {
		node {$\left[\begin{array}{*5{@{}wc{2mm}@{}}}
			\blank&\blank&\blank&\blank&\blank\\
			\blank&\blank&     0&\blank&\blank\\
			\blank&\blank&\blank&\blank&\blank
		\end{array}\right]$}
		child {
			node {$\left[\begin{array}{*5{@{}wc{2mm}@{}}}
				\blank&0&\blank&\blank&\blank\\
				\blank&\blank&0&\blank&\blank\\
				\blank&\blank&\blank&\blank&\blank
			\end{array}\right]$}
			child {
				node {$\left[\begin{array}{*5{@{}wc{2mm}@{}}}
					\blank&0&0&\blank&\blank\\
					\blank&\blank&0&\blank&\blank\\
					\blank&\blank&\blank&\blank&\blank
				\end{array}\right]$}
				edge from parent node[left] {$\vartheta_{1, 3} \gets 0$};
			}
			child {
				node { \scalebox{5}{$\times$} }
				edge from parent node[right] {$\vartheta_{1, 3} \gets 1$};
			}
			edge from parent node[left] {$\vartheta_{1, 2} \gets 0$};
		}
		child {
			node {$\left[\begin{array}{*5{@{}wc{2mm}@{}}}
				\blank&1&\blank&\blank&\blank\\
				\blank&\blank&0&\blank&\blank\\
				\blank&\blank&\blank&\blank&\blank
			\end{array}\right]$}
			child {
				node {$\left[\begin{array}{*5{@{}wc{2mm}@{}}}
					\blank&1&0&\blank&\blank\\
					\blank&\blank&0&\blank&\blank\\
					\blank&\blank&\blank&\blank&\blank
				\end{array}\right]$}
				edge from parent node[left] {$\vartheta_{1, 3} \gets 0$};
			}
			child {
				node {$\left[\begin{array}{*5{@{}wc{2mm}@{}}}
					\blank&1&1&\blank&\blank\\
					\blank&\blank&0&\blank&\blank\\
					\blank&\blank&\blank&\blank&\blank
				\end{array}\right]$}
				edge from parent node[right] {$\vartheta_{1, 3} \gets 1$};
			}
			edge from parent node[right] {$\vartheta_{1, 2} \gets 1$};
		}
		edge from parent node[left] {$\vartheta_{2, 3} \gets 0$};
	}
	child {
		node {$\left[\begin{array}{*5{@{}wc{2mm}@{}}}
			\blank&\blank&\blank&\blank&\blank\\
			\blank&\blank&     1&\blank&\blank\\
			\blank&\blank&\blank&\blank&\blank
		\end{array}\right]$}
		child{
			node {$\left[\begin{array}{*5{@{}wc{2mm}@{}}}
				\blank&0&\blank&\blank&\blank\\
				\blank&\blank&1&\blank&\blank\\
				\blank&\blank&\blank&\blank&\blank
			\end{array}\right]$}
			child {
				node {$\left[\begin{array}{*5{@{}wc{2mm}@{}}}
					\blank&0&0&\blank&\blank\\
					\blank&\blank&1&\blank&\blank\\
					\blank&\blank&\blank&\blank&\blank
				\end{array}\right]$}
				edge from parent node[left] {$\vartheta_{1, 3} \gets 0$};
			}
			child {
				node {$\left[\begin{array}{*5{@{}wc{2mm}@{}}}
					\blank&0&1&\blank&\blank\\
					\blank&\blank&1&\blank&\blank\\
					\blank&\blank&\blank&\blank&\blank
				\end{array}\right]$}
				edge from parent node[right] {$\vartheta_{1, 3} \gets 1$};
			}
			edge from parent node[left] {$\vartheta_{1, 2} \gets 0$};
		}
		child{
			node {$\left[\begin{array}{*5{@{}wc{2mm}@{}}}
				\blank&1&\blank&\blank&\blank\\
				\blank&\blank&1&\blank&\blank\\
				\blank&\blank&\blank&\blank&\blank
			\end{array}\right]$}
			child {
				node {$\left[\begin{array}{*5{@{}wc{2mm}@{}}}
					\blank&1&0&\blank&\blank\\
					\blank&\blank&1&\blank&\blank\\
					\blank&\blank&\blank&\blank&\blank
				\end{array}\right]$}
				edge from parent node[left] {$\vartheta_{1, 3} \gets 0$};
			}
			child {
				node {$\left[\begin{array}{*5{@{}wc{2mm}@{}}}
					\blank&1&1&\blank&\blank\\
					\blank&\blank&1&\blank&\blank\\
					\blank&\blank&\blank&\blank&\blank
				\end{array}\right]$}
				edge from parent node[right] {$\vartheta_{1, 3} \gets 1$};
			}
			edge from parent node[right] {$\vartheta_{1, 2} \gets 1$};
		}
		edge from parent node[right] {$\vartheta_{2, 3} \gets 1$};
	}
;
\end{tikzpicture}
\caption{Branch-and-bound tree for the NDB problem with pruned nodes by 
isomorphism pruning.}
\label{fig:branching:isomorphismpruning}
\end{figure}

Note that isomorphism pruning is a pruning method,
which means that it does not find reductions. 
However, isomorphism pruning can be enhanced by fixing rules
that allow to find additional variable fixings early on in the
branch-and-bound tree as we discuss next.

\subsubsection{Orbital fixing} 
\label{sec:orbitalfixing}
Orbital fixing (\orbitalfixing) refers to a family of variable domain
reductions (\VDRs), whose common ground is to fix variables within orbits of
already fixed variables.
The exact definition of \orbitalfixing differs between different 
authors~\cite{Margot2003,OstrowskiEtAl2011}.
The main difference is whether fixings found by branching decisions
are distinguished from fixings found by orbital fixing.
We describe the variant~\cite[Theorem~3]{OstrowskiEtAl2011}, which is
compatible with isomorphism pruning.
Let~$\beta \in \bnbnodes$.
The group consisting of all permutations that stabilize the~1-branchings
up to node~$\beta$ is denoted by~$\Delta^\beta \define \stab{\Gamma}{B_1^\beta}
\define \{ \gamma \in \Gamma : \gamma(B_1^\beta) = B_1^\beta \}$.
\begin{theorem}[Orbital Fixing]
  Let~$\beta \in \bnbnodes$.
  If~$i \in B_0^\beta$, all variables in the $\Delta^\beta$-orbit
  of~$i$ can be fixed to~0.
\end{theorem}
Figure~\ref{fig:branching:orbitalfixing} shows the branch-and-bound
tree for applying this orbital fixing rule to the running example.
Note that, if up to node $\beta$ no variables are branched to one,
$\Delta^\beta$ corresponds to the symmetry group $\Gamma$.
This means that for zero-branchings its whole orbit of $\Gamma$ 
(the corresponding row in $\vartheta$)
can be fixed to zero.
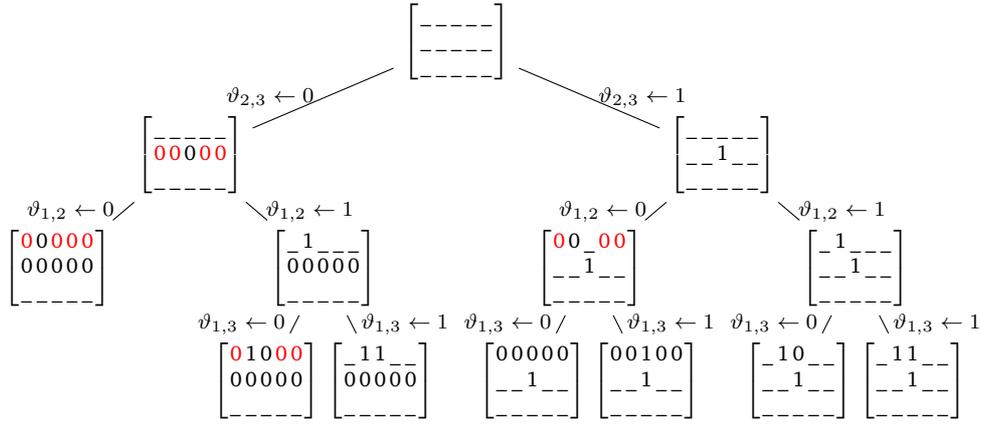
\begin{figure}[!tbp]
\centering
\begin{tikzpicture}[
	x=1cm, y=1cm, level distance=1.5cm,
	level 1/.style={sibling distance=7cm},
	level 2/.style={sibling distance=3.5cm},
	level 3/.style={sibling distance=1.5cm},
	font=\footnotesize
]

\node {$\left[\begin{array}{*5{@{}wc{2mm}@{}}}
		\blank&\blank&\blank&\blank&\blank\\
		\blank&\blank&\blank&\blank&\blank\\
		\blank&\blank&\blank&\blank&\blank
	\end{array}\right]$}
	child {
		node {$\left[\begin{array}{*5{@{}wc{2mm}@{}}}
			\blank&\blank&\blank&\blank&\blank\\
			\red 0&\red 0&     0&\red 0&\red 0\\
			\blank&\blank&\blank&\blank&\blank
		\end{array}\right]$}
		child {
			node {$\left[\begin{array}{*5{@{}wc{2mm}@{}}}
				\red 0&0&\red 0&\red 0&\red 0\\
				     0&     0&0&     0&     0\\
				\blank&\blank&\blank&\blank&\blank
			\end{array}\right]$}
			edge from parent node[left] {$\vartheta_{1, 2} \gets 0$};
		}
		child {
			node {$\left[\begin{array}{*5{@{}wc{2mm}@{}}}
				\blank&1&\blank&\blank&\blank\\
				     0&     0&0&     0&     0\\
				\blank&\blank&\blank&\blank&\blank
			\end{array}\right]$}
			child {
				node {$\left[\begin{array}{*5{@{}wc{2mm}@{}}}
					\red 0&     1&     0&\red 0&\red 0\\
					     0&     0&     0&     0&     0\\
					\blank&\blank&\blank&\blank&\blank
				\end{array}\right]$}
				edge from parent node[left] {$\vartheta_{1, 3} \gets 0$};
			}
			child {
				node {$\left[\begin{array}{*5{@{}wc{2mm}@{}}}
					\blank&     1&     1&\blank&\blank\\
					     0&     0&     0&     0&     0\\
					\blank&\blank&\blank&\blank&\blank
				\end{array}\right]$}
				edge from parent node[right] {$\vartheta_{1, 3} \gets 1$};
			}
			edge from parent node[right] {$\vartheta_{1, 2} \gets 1$};
		}
		edge from parent node[left] {$\vartheta_{2, 3} \gets 0$};
	}
	child {
		node {$\left[\begin{array}{*5{@{}wc{2mm}@{}}}
			\blank&\blank&\blank&\blank&\blank\\
			\blank&\blank&     1&\blank&\blank\\
			\blank&\blank&\blank&\blank&\blank
		\end{array}\right]$}
		child{
			node {$\left[\begin{array}{*5{@{}wc{2mm}@{}}}
				\red 0&     0&\blank&\red 0&\red 0\\
				\blank&\blank&     1&\blank&\blank\\
				\blank&\blank&\blank&\blank&\blank
			\end{array}\right]$}
			child {
				node {$\left[\begin{array}{*5{@{}wc{2mm}@{}}}
					     0&     0&     0&     0&     0\\
					\blank&\blank&     1&\blank&\blank\\
					\blank&\blank&\blank&\blank&\blank
				\end{array}\right]$}
				edge from parent node[left] {$\vartheta_{1, 3} \gets 0$};
			}
			child {
				node {$\left[\begin{array}{*5{@{}wc{2mm}@{}}}
					0&0&1&0&0\\
					\blank&\blank&1&\blank&\blank\\
					\blank&\blank&\blank&\blank&\blank
				\end{array}\right]$}
				edge from parent node[right] {$\vartheta_{1, 3} \gets 1$};
			}
			edge from parent node[left] {$\vartheta_{1, 2} \gets 0$};
		}
		child{
			node {$\left[\begin{array}{*5{@{}wc{2mm}@{}}}
				\blank&1&\blank&\blank&\blank\\
				\blank&\blank&1&\blank&\blank\\
				\blank&\blank&\blank&\blank&\blank
			\end{array}\right]$}
			child {
				node {$\left[\begin{array}{*5{@{}wc{2mm}@{}}}
					\blank&1&0&\blank&\blank\\
					\blank&\blank&1&\blank&\blank\\
					\blank&\blank&\blank&\blank&\blank
				\end{array}\right]$}
				edge from parent node[left] {$\vartheta_{1, 3} \gets 0$};
			}
			child {
				node {$\left[\begin{array}{*5{@{}wc{2mm}@{}}}
					\blank&1&1&\blank&\blank\\
					\blank&\blank&1&\blank&\blank\\
					\blank&\blank&\blank&\blank&\blank
				\end{array}\right]$}
				edge from parent node[right] {$\vartheta_{1, 3} \gets 1$};
			}
			edge from parent node[right] {$\vartheta_{1, 2} \gets 1$};
		}
		edge from parent node[right] {$\vartheta_{2, 3} \gets 1$};
	}
;
\end{tikzpicture}
\caption{Branch-and-bound tree for the NDB problem with fixings by
orbital fixing.}
\label{fig:branching:orbitalfixing}
\end{figure}

Since~\cite{OstrowskiEtAl2011} does not distinguish variables fixed 
to~1 by branching or other decisions, $\Delta^\beta$ can be replaced by all
permutations that stabilize the variables that are fixed to~1 (opposed to
just branched to be~1).
Note that neither definition of~$\Delta^\beta$ contains the other, i.e.,
neither version of \orbitalfixing dominates the other in terms of the
number of fixings that can be found.
Another variant of \orbitalfixing that also finds $1$-fixings is
presented in~\cite{ostrowski2009symmetry}, see
also~\cite{pfetsch2019computational}.

\subsection{Further symmetry handling methods}

Liberti and Ostrowski~\cite{LibertiOstrowski2014} as well as
Salvagnin~\cite{Salvagnin2018} present symmetry handling inequalities that
can be derived from the Schreier-Sims table of group.
Further symmetry handling inequalities are described by Liberti~\cite{Liberti2012a}.
In contrast to the constraints from Section~\ref{sec:shcbin}, they are also able to
handle symmetries of non-binary variables; their symmetry handling effect
is limited though.
Another class of inequalities, so-called orbital conflict inequalities, have
been proposed in~\cite{LinderothEtAl2021}.
Moreover, symmetry handling inequalities for specific problem classes are
discussed, among others,
by~\cite{GhoniemSherali2011,Hojny2020,KaibelPfetsch2008,MendezDiazZabala2001,sherali2001models}.

Besides the propagation approaches discussed above, also tailored
complete algorithms that can handle special cyclic groups
exist~\cite{doornmalenhojny2022cyclicsymmetries}.
Moreover, Ostrowski~\cite{ostrowski2009symmetry} presents smallest-image
fixing, a propagation algorithm for binary variables.
Instead of exploiting symmetries in a propagation framework, symmetries can
also be handled by tailored branching
rules~\cite{OstrowskiEtAl2011,OstrowskiAnjosVannelli2015}.
Furthermore, orbital shrinking~\cite{FischettiLiberti2012} is a method that
handles symmetries by aggregating variables contained in a common orbit,
which results in a relaxation of the problem.
Finally, core points~\cite{Bodi2013,Herr2013,HerrRehnSchuermann2013} can be
used to restrict the feasible region of problems to a subset of solutions.
This latter approach does not coincide with lexicographically maximal
representatives.

\section{Unified framework for symmetry handling}
\label{sec:framework}

As the literature review shows, different symmetry handling methods use
different paradigms to derive symmetry reductions.
For instance, \SHCs remove symmetric solutions from the initial problem
formulation, whereas methods such as orbital fixing remove symmetric
solutions based on the branching history.
At first glance, these methods thus are not necessarily compatible.

To overcome this seeming incompatibility, we present a unified framework
for symmetry handling that easily allows to check whether symmetry handling
methods are compatible.
It turns out that, via our framework, isomorphism pruning and \orbitalfixing
can be made compatible with a variant of \lexfix.
Moreover, in contrast to many symmetry handling methods discussed in the
literature, our framework also applies to non-binary problems and is not
restricted to permutation symmetries.
Before we present our framework in Section~\ref{sec:theframework}, it will
be useful to first provide an interpretation of isomorphism pruning through
the lens of symmetry handling constraints.

\subsection{Isomorphism pruning and orbital fixing revisited}
\label{sec:isopruneRevisited}

Let~$\beta \in \bnbnodes$ be a node at depth~$m$ and let~$y$ be the
incidence vector of 1-branching decisions as described in
Section~\ref{sec:isomorphismpruning}.
Due to Theorem~\ref{thm:isoprune}, node~$\beta$ can be pruned by
isomorphism if solution vector~$y$ 
violates~$\restrict{\pi_\beta(x)}{[m]} \succeq
\restrict{\pi_\beta(\gamma(x))}{[m]}$ for some~$\gamma \in \Gamma$.
The latter condition looks very similar to classical \SHCs, however, there
are some differences:
the variable order is changed via~$\pi_\beta$, not all variables are
present in this constraint due to the restriction, and most importantly,
every node has a potentially different reordering and restriction.
Nevertheless, these modified \SHCs can be used to remove all symmetries
from a binary problem in the sense that, for every solution~$x$ of a binary
problem, there exists exactly one node of~$\bnbtree$ at depth~$n$ that
contains a symmetric counterpart of~$x$,
see~\cite[Thm.~4.5]{ostrowski2009symmetry}.

Based on the modified \SHCs, it is easy to show that isomorphism pruning
and orbital fixing are compatible, provided one can show that both methods
are compatible with the modified \SHCs.
\begin{lemma}
  Let~$\beta \in \bnbnodes$ be a node at depth~$m$.
  If~$\beta$ gets pruned by isomorphism, there is no~$x \in \binary^n$
  that is feasible for the subproblem at~$\beta$ and that satisfies
  $\restrict{\pi_\beta(x)}{[m]} \succeq
  \restrict{\pi_\beta(\gamma(x))}{[m]}$ for all~$\gamma \in \Gamma$.
\end{lemma}
\begin{proof}
  As in Section~\ref{sec:isomorphismpruning}, let~$y \in \binary^n$ be such
  that~$y_i = 1$ if and only if~$i \in B_1^\beta$.
  If \isomorphismpruning prunes~$\beta$, there is~$\gamma\in\Gamma$
  with~$\restrict{\pi_\beta(y)}{[m]} \prec \restrict{\pi_\beta(\gamma(y))}{[m]}$.
  As the first $m$ entries of $\pi_\beta(y)$ are branching variables 
  and the remaining entries are~0, we find~$\restrict{\pi_\beta(\gamma(x))}{[m]} 
  \geq \restrict{\pi_\beta(\gamma(y))}{[m]}$ 
  (componentwise) for each~$x \in \feasregion(\beta)$.
  Thus, $\restrict{\pi_\beta(x)}{[m]} = \restrict{\pi_\beta(y)}{[m]} 
  \prec \restrict{\pi_\beta(\gamma(y))}{[m]}
  \leq \restrict{\pi_\beta(\gamma(x))}{[m]}$,
  which means every solution in~$\feasregion(\beta)$
  violates~$\restrict{\pi_\beta(x)}{[m]} \succeq
  \restrict{\sigma_\beta(\gamma(x))}{[m]}$.
\end{proof}
\begin{lemma}
  Let~$\beta \in \bnbnodes$ be a node at depth~$m$.
  Every fixing found by~\orbitalfixing at node~$\beta$ is implied by
  $\restrict{\pi_\beta(x)}{[m]} \succeq
  \restrict{\pi_\beta(\gamma(x))}{[m]}$ for all~$\gamma \in \Gamma$.
\end{lemma}
\begin{proof}
  Assume \orbitalfixing is not compatible with~$\restrict{\pi_\beta(x)}{[m]} \succeq
  \restrict{\pi_\beta(\gamma(x))}{[m]}$, $\gamma \in \Gamma$.
  Then, there exists a node~$\beta \in \bnbnodes$, a solution~$\bar{x} \in
  \feasregion(\beta)$ that satisfies~$\restrict{\pi_\beta(\bar{x})}{[m]} \succeq
  \restrict{\pi_\beta(\gamma(\bar{x}))}{[m]}$ for all $\gamma \in \Gamma$, and
  an index~$j \in [m]$
  with $i_j \in B_0^\beta$
  such that~$\bar{x}_{\ell} = 1$ for some~$\ell$ in
  the~$\Delta^\beta$-orbit of~$i_j$.
  Suppose~$j$ is minimal.

  Since~$\ell$ is contained in the~$\Delta^\beta$-orbit of~$i_j$, there
  exists~$\gamma \in \Delta^\beta$ with~$\gamma(\ell) = i_j$.
  By definition of~$\Delta^\beta$, 
  for all~$k \in B_1^\beta$, $\gamma(k) \in B^\beta_1$.
  Moreover, $\pi_\beta(\gamma(\bar x))_k = \bar x_{\gamma^{-1}(i_k)} = 0$ 
  for all $k \in [j-1]$ with $i_k \in B_0^\beta$,
  because~$j$ is selected minimally.
  Consequently, since $B_0^\beta \cup B_1^\beta = [m]$ holds,
  $\pi_\beta(\bar x)$ and~$\pi_\beta(\gamma(\bar x))$
  coincide on the first~$j-1$
  entries, and~$1 = \bar x_\ell = \bar x_{\gamma^{-1}(i_j)} = 
  \pi_\beta(\gamma(\bar x))_j 
  > \pi_\beta(\bar x))_j = \bar x_{i_j} = 0$.
  That is, $\restrict{\pi_\beta(\bar x)}{[m]} 
  \prec \restrict{\pi_\beta(\gamma(\bar x))}{[m]}$,
  contradicting that~$\bar{x}$ satisfies all \SHCs.
  \orbitalfixing is thus compatible with the \SHCs.
\end{proof}

Isomorphism pruning and orbital fixing are thus compatible.
While isomorphism pruning can become active as soon as one can show that no
lexicographically maximal solution w.r.t.\ the modified \SHCs is feasible
at a node~$\beta$, orbital fixing might not be able to find all symmetry
related variable reductions.
\begin{example}\label{ex:OFweak}
Let~$\Gamma \leq \symmetricgroup{4}$ be generated by a cyclic
right shift, i.e., the non-trivial permutations in~$\Gamma$ are~$\gamma_1 =
(1,2,3,4)$, $\gamma_2 = (1,3)(2,4)$, and~$\gamma_3 = (1,4,3,2)$.
Consider the branch-and-bound tree in Figure~\ref{fig:exBB}.
At node~$\beta_0$, no reductions can be found by \orbitalfixing as
no proper shift fixes the $1$-branching variable~$x_3$.
For~$\gamma_1$, the SHC~$\restrict{\pi_{\beta_0}(x)}{[2]} \succeq
\restrict{\pi_{\beta_0}(\gamma(x))}{[2]}$ reduces to
$(x_3, x_4) \succeq (x_2, x_3)$.
Due to the variable bounds at~$\beta_0$, the constraint simplifies
to~$(1,0) \succeq (x_2,1)$.
This constraint is violated if $x_2$ has value~1,
so $x_2$ can be fixed to~0.
\end{example}

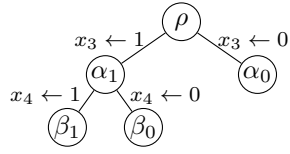
\begin{figure}[!tbp]
\centering
\begin{tikzpicture}[x=1cm, y=.7cm]

\node (b) at (0,0) [bbnode] {$\rho$};

\node (l) at (-1,-1) [bbnode] {$\alpha_1$};
\node (r) at (1,-1) [bbnode] {$\alpha_0$};

\node (ll) at (-1.5,-2) [bbnode] {$\beta_1$};
\node (lr) at (-0.5,-2) [bbnode] {$\beta_0$};

\draw[-] (b) to node [left, near start]
{\footnotesize{$x_3\gets1$}}  (l);
\draw[-] (b) to node [right,near start]
{\footnotesize{$x_3\gets0$}}  (r);
\draw[-] (l) to node [left, near start]
{\footnotesize{$x_4\gets1$}}  (ll);
\draw[-] (l) to node [right,near start]
{\footnotesize{$x_4\gets0$}}  (lr);
\end{tikzpicture}
\caption{Branch-and-bound tree for Example~\ref{ex:OFweak}.}
\label{fig:exBB}
\end{figure}

Consequently, symmetry handling by isomorphism pruning and orbital fixing
can be improved by identifying further symmetry handling methods that are
compatible with the modified \SHCs.

\subsection{The framework}
\label{sec:theframework}

In this section, we present our unified framework for symmetry handling
with the following goals:
It should
\begin{enumerate}[label={{(G\arabic*)}},ref={G\arabic*}]
\item\label{G1} allow to check whether different symmetry
  handling methods are compatible.
  In particular, it should ensure compatibility of \lexfix, isomorphism
  pruning, and \orbitalfixing.
\item\label{G2} generalize the modified \SHCs by
  Ostrowski~\cite{ostrowski2009symmetry}.
\item\label{G3} apply to general variable types and general symmetries (not
  necessarily permutations).
\end{enumerate}
To achieve these goals, we define a more general class of
\SHCs~$\sigma_\beta(x) \succeq \sigma_\beta(\gamma(x))$, where~$\gamma \in
\Gamma$ and~$\beta \in \bnbnodes$, that are not necessarily based on
branching decisions.

Let~$\feasregion \subseteq \R^n$ and let~$f\colon \feasregion \to \R$ be
such that~$\optval(f,\feasregion)$ can be solved by (spatial)
branch-and-bound.
Let~$\Gamma$ be a group of symmetries of~$\optval(f,\feasregion)$.
Let~$\bnbtree = (\bnbnodes, \bnbedges)$ be a branch-and-bound tree and
let~$\beta \in \bnbnodes$.
In our modified \SHCs, the map~$\sigma_\beta\colon \R^n \to \R^{m_\beta}$
will be parameterized via a permutation~$\pi_\beta
\in \symmetricgroup{n}$, a symmetry~$\varphi_\beta \in \Gamma$, and an
integer~$m_\beta \in \{0, \dots, n\}$ 
as~$\sigma_\beta(\cdot) 
\define \restrict{\left( 
\pi_\beta \circ \varphi_\beta (\cdot) 
\right)}{[m_\beta]}$.
As in Ostrowski's approach, $\pi_\beta$ selects a variable ordering
and~$m_\beta$ allows to restrict the \SHCs to a subset of variables.
In contrast to~\cite{ostrowski2009symmetry}, however, $\pi_\beta$ does not
necessarily correspond to the branching order.
Moreover, $\varphi_\beta$ provides more degrees of freedom as it allows to
change the variable order imposed by~$\pi_\beta$.
We refer to the structure~$(m_\beta, \pi_\beta, \varphi_\beta)_{\beta \in \bnbnodes}$ 
as a \emph{symmetry prehandling structure} for $\bnbtree$.
Note that this definition already achieves goal~\eqref{G2} by
setting~$\varphi_\beta = \id$, using the same~$\pi_\beta$ as in
Section~\ref{sec:isopruneRevisited}, and setting~$m_\beta$ to be the number
of different branching variables in node~$\beta$.

\begin{theorem}
  \label{thm:main}
  Let~$\feasregion \subseteq \R^n$ and let~$f\colon \feasregion \to \R$ be
  such that~$\optval(f,\feasregion)$ can be solved by (spatial)
  branch-and-bound.
  Let~$\Gamma$ be a finite group of symmetries of~$\optval(f,\feasregion)$.
  Suppose that the branch-and-bound method used for
  solving~$\optval(f,\feasregion)$ generates a finite full
  \bb-tree~$\bnbtree = (\bnbnodes, \bnbedges)$.
  For each node~$\beta \in \bnbnodes$,
  let~$(m_\beta,\pi_\beta,\varphi_\beta) \in [n]_0 \times \symmetricgroup{n}
  \times \Gamma$.
  Let~$\sigma_\beta(\cdot) = \restrict{(\pi_\beta \circ 
  \varphi_\beta(\cdot))}{[m_\beta]}$.
  Suppose that we enforce, for every~$\beta \in \bnbnodes$,
  \begin{equation}
    \label{eq:main}
    \sigma_\beta(x) \succeq \sigma_\beta (\gamma(x))\
    \text{for all}\
    \gamma \in \Gamma.
  \end{equation}
If~$(m_\beta,\pi_\beta,\varphi_\beta)$
satisfies so-called \emph{correctness conditions}~%
\eqreffromto{cond:ffunc}{cond:permutationcondition}
for all nodes~$\beta \in \bnbnodes$:
\begin{enumerate}[
label={{(C\arabic*)}},
ref={C\arabic*}, itemsep={.5em}
]
\item
\label{cond:ffunc}
If $\beta$ has a parent $\alpha \in \bnbnodes$,
then $m_\beta \geq m_\alpha$
and for all $i \leq m_\alpha$ and $x \in \R^n$
holds $\pi_\alpha(x)_i = \pi_\beta(x)_i$;

\item
\label{cond:symmetricallypermute}
If $\beta$ has a parent $\alpha \in \bnbnodes$,
then $\varphi_\beta =\varphi_\alpha \circ \psi_\alpha$
for some
\[
  \psi_\alpha \in \stab{\Gamma}{\feasregion(\alpha)}
  \define \{ \gamma \in \Gamma :
  \feasregion(\alpha) = \gamma(\feasregion(\alpha)) \};
\]

\item
\label{cond:sibl}
If $\beta$ has a sibling $\beta' \in \bnbnodes$,
then $m_\beta = m_{\beta'}$, $\pi_\beta = \pi_{\beta'}$,
and $\varphi_\beta = \varphi_{\beta'}$, i.e., $\psi_\alpha$ in \eqref{cond:symmetricallypermute} does not depend on $\beta$;

\item
\label{cond:permutationcondition}
If $\beta$ has a feasible solution $x \in \feasregion(\beta)$,
then for all permutations $\xi \in \Gamma$
with $\sigma_\beta(x) = \sigma_\beta(\xi(x))$
also the permuted solution $\xi(x)$ is feasible in $\feasregion(\beta)$;
\end{enumerate}
then, for each~$\tilde{x} \in \feasregion$,
there is exactly one leaf~$\nu$ of the \bb-tree con\-taining a solution
symmetric to~$\tilde{x}$, i.e., for which there is~$\xi \in \Gamma$
with~$\xi(\tilde{x}) \in \feasregion(\nu)$.
\end{theorem}

Before we apply and prove this theorem, we interpret the correctness
conditions and provide some implications and consequences.
We start with the latter.

\begin{itemize}
\item Enforcing~\eqref{eq:main} handles
  symmetries by excluding feasible solutions from the search space while
  guaranteeing that exactly one representative solution per class of
  symmetric solutions remains feasible (recall that~$\bnbtree$ does not
  prune nodes by bound).
  Note that by enforcing~\eqref{eq:main}, symmetry reductions can only take
  place on variables ``seen'' by~$\sigma_\beta(x) \succeq
  \sigma_\beta(\gamma(x))$ for some~$\gamma\in\Gamma$.
  We stress that it is not immediate how~\eqref{eq:main} can be enforced
  efficiently.
  We will turn to this question in Section~\ref{sec:cast}.

\item If we prune nodes by bound, \eqref{eq:main} still can be used to
  handle symmetries.
  But not necessarily all~$x \in \feasregion$ have
  a symmetric counterpart feasible at some leaf (e.g., if~$x$ is
  suboptimal).

\item If not all constraints of type~\eqref{eq:main} are completely
enforced, we still find valid symmetry reductions, but not necessarily
exactly one representative solution.

\item If different symmetry handling methods can be expressed in terms
of~\eqref{eq:main} having the same choice of the symmetry prehandling 
structure~$(m_\beta, \pi_\beta, \varphi_\beta)_{\beta \in \bnbnodes}$,
then both symmetry handling methods can be applied at the same time,
i.e., they are \emph{compatible}.

\item In practice, \bnb is enhanced by cutting planes or domain
  propagation such as reduced cost fixing.
  Both also work in our framework if their reductions are
  \emph{symmetry compatible}, i.e., if, for~$\beta \in \bnbnodes$, the
  domain of a variable~$x_i$ is reduced, the same reduction can be applied
  to all symmetric variables w.r.t.\ symmetries at~$\beta$.
  Margot~\cite[Section~4]{Margot2003} discusses this in detail for
  \isomorphismpruning. He refers to this as \emph{strict setting
    algorithms}.

\item For spatial branch-and-bound, the children of a node~$\alpha$ do not
  necessarily partition~$\feasregion(\alpha)$ (the feasible regions of
  children can overlap on their boundary).
  In this case, \eqref{eq:main} can still be used to handle symmetries, but
  there might exist several leaves containing a symmetrically equivalent
  solution.
\end{itemize}

\begin{remark}
\label{remark:improper}
As propagating \SHCs cuts off feasible solutions, 
such propagations are not symmetry-compatible.
Therefore, we consider \SHC reductions in our framework as special
branching decisions, called \emph{improper}:
For a \SHC reduction $C \subseteq \R^n$ 
at node~$\beta \in \bnbnodes$,
two children $\omega, \omega'$ are introduced
with $\feasregion(\omega) = \feasregion(\beta) \cap C$
and~
$\feasregion(\omega') = \feasregion(\beta) \setminus \feasregion(\omega)$.
Node $\omega'$ can then be pruned by symmetry.
Complementing this, traditional (standard) branching decisions are called 
\emph{proper}.
\end{remark}

\paragraph{Interpretation}
Theorem~\ref{thm:main} iteratively builds
\SHCs~$\sigma_\beta(x) \succeq \sigma_\beta(\gamma(x))$ that do not
necessarily build upon a common lexicographic order for different
nodes~$\beta \in \bnbnodes$.
The map~${\sigma_\beta(\cdot) = \restrict{(\pi_\beta \circ 
\varphi_\beta(\cdot))}{[m_\beta]}}$
accepts an~$n$-dimensional vector,
considers a symmetrically equivalent representative solution hereof
($\varphi_\beta$),
reorders its entries ($\pi_\beta$), and afterwards restricts
them to the first~$m_\beta$ coordinates.
This way, $\sigma_\beta$ selects~$m_\beta$ expressions (and their images) 
that
appear in the \SHCs~\eqref{eq:main}.
To ensure that consistent \SHCs are derived, sufficient information
needs to be inherited to a node's children in the \bb-tree, which is
achieved as follows.

For the ease of explanation, let us first assume~$\varphi_\beta$ is the
identity~$\id$.
Then, \eqref{cond:ffunc} guarantees that a child has not less
information than its parent.
Moreover, siblings must not be too different, i.e., new information at one
child also needs to be known to its siblings~\eqref{cond:sibl}.
\eqref{cond:permutationcondition} ensures that if two 
solutions~$x$ and~$\xi(x)$ appear identical for the \SHCs in
the sense~$\sigma_\beta(x) = \sigma_\beta(\xi(x))$, feasibility of~$x$
should imply feasibility of~$\xi(x)$.
In other words, if $x$ and $\xi(x)$ 
are identical with respect to $\sigma_\beta$,
it may not be that one solution is feasible at $\beta$ 
while the other solution is not.

Conditions~\eqref{cond:ffunc},
\eqref{cond:sibl}, and~\eqref{cond:permutationcondition}
describe how $\sigma_\beta(x)$ ``grows'' as nodes $\beta$ follow a rooted path,
and that siblings are handled in the same way.
If $\varphi_\beta = \id$, 
for a node $\beta$ with ancestor $\mu$
all variables and expressions of $\sigma_\mu(x)$
also occur in the first $m_\mu$ elements of $\sigma_\beta(x)$.
Condition~\eqref{cond:symmetricallypermute} allows for more 
flexibility in this.
Let $\alpha$ be the parent of $\beta$.
If there is a symmetry $\gamma \in \Gamma$
that leaves the feasible region of $\alpha$ invariant
(i.e., $\feasregion(\alpha) = \gamma(\feasregion(\alpha))$),
one can choose to handle the symmetries considering the
symmetrically equivalent solution space
as of node $\beta$. 
This degree of freedom might help a solver to find more symmetry reductions
in comparison to just ``growing'' the considered representatives.
For example, in Figure~\ref{fig:branching:orbitopalfixing}
at node $\alpha$ 
with~$(\vartheta_{2,3},\vartheta_{1,2},\vartheta_{1,3}) \gets (0,1,0)$
the feasible region $\feasregion(\alpha)$
is identical when permuting the first two columns
or the last three columns.
Suppose that one branches next on variable $\vartheta_{3,3}$,
then the zero-branch will find two reductions
(namely $\vartheta_{3,4},\vartheta_{3,5} \gets 0$)
and the one-branch will find no reductions.
If the solver has a preference to reduce the discrepancy between the number
of reductions found over the siblings, 
one could exchange column 3 and 4 for the sake of symmetry handling.
Effectively, this moves the branching variable
to the fourth column. Applying orbitopal fixing on the matrix where 
these columns are exchanged leads to one fixing in either child.

\paragraph{Examples}
Let~$\bnbtree = (\bnbnodes, \bnbedges)$ be a \bb-tree, in which each
branching decision partitions the domain of exactly one variable.
We will show that there are many possible symmetry prehandling structures%
~$(m_\beta, \pi_\beta, \varphi_\beta)_{\beta \in \bnbnodes}$ 
that satisfy the correctness conditions of Theorem~\ref{thm:main}.
Hence, this gives many degrees of freedom to handle symmetries.
In the following, we discuss choices that resemble three  symmetry handling
techniques: static \SHCs, Ostrowksi's branching variable ordering, and a
variant of orbitopal fixing that 
is more flexible than the setting of Bendotti~\etal.

\begin{example}[Static \SHCs]
  \label{ex:staticsetting}
  The static \SHCs~$x \succeq \gamma(x)$ for all~$\gamma \in \Gamma$ can be
  derived in our framework by setting, for each~$\beta \in \bnbnodes$, the
  parameters~$m_\beta = n$, $\pi_\beta = \varphi_\beta = \psi_\beta = \id$.
  \eqreffromto{cond:ffunc}{cond:sibl} are satisfied trivially.
  As any~$x \in \R^n$ satisfies
  $\sigma_\beta(x) = \restrict{\pi_\beta \varphi_\beta(x)}{[n]} = x$,
  we find $\sigma_\beta(x) = \sigma_\beta \gamma(x)$
  if and only if~$x = \gamma(x)$. Hence, also
  \eqref{cond:permutationcondition} holds.
\end{example}

Next, we resemble Ostrowski's rank for binary variables and generalize
it to arbitrary variable types.
In the latter case, only considering the branching order
is not sufficient as one might branch several times on the
same variable.

\begin{example}[Branching-based]
  \label{ex:vardynamic}
  Let~$\beta \in \bnbnodes$.
  If~$\beta$ is the root node, let~$m_\beta = 0$, i.e., $\sigma_\beta$ is 
  void.
  Otherwise, let~$\alpha$ be the parent of~$\beta$.
  If~$\beta$ arises from~$\alpha$ by a proper branching decision on
  variable~$x_{\ihat}$ and~$\ihat$ has not been used for branching
  before, i.e., $\ihat \notin (\pi_\beta \varphi_\beta)^{-1}([m_\beta])$,
  then set~$m_\beta = m_\alpha + 1$, $\varphi_\beta = \id$ and
  select~$\pi_\beta \in \symmetricgroup{n}$ with~$\pi_\beta(i) =
  \pi_\alpha(i)$ for $i \leq m_\alpha$ and $\pi_\beta(m_\beta) = \ihat$.
  Otherwise, inherit the symmetry prehandling structure from~$\alpha$,
  i.e.,
  $\pi_\beta = \pi_\alpha$, $m_\beta = m_\alpha$, and $\varphi_\beta = 
  \id$.
\end{example}

\begin{proof}[Example~\ref{ex:vardynamic} satisfies 
\eqreffromto{cond:ffunc}{cond:permutationcondition}]
\eqreffromto{cond:ffunc}{cond:sibl} hold trivially.
To show~\eqref{cond:permutationcondition},
let~$x \in \feasregion(\beta)$
and $\xi \in \Gamma$
such that~$\sigma_\beta(x) = \sigma_\beta(\xi(x))$.
By definition of~$(m_\beta,\pi_\beta,\varphi_\beta)$, $\sigma_\beta(x)$
restricts~$x$ onto all (resorted) variables used for branching up to
node~$\beta$.
To show~\eqref{cond:permutationcondition}, note that the
feasible region $\feasregion(\beta)$ is the intersection
of
\begin{enumerate*}[label={(\roman*)},
ref={(\roman*)}]
\item \label{ex:vardynamic:pr:1}
$\feasregion$,
\item \label{ex:vardynamic:pr:2}
proper branching decisions, and
\item \label{ex:vardynamic:pr:3}
symmetry reductions due to~\eqref{eq:main}.
\end{enumerate*}

It is thus sufficient to show~$\xi(x)$ is contained in each of these
sets.
Since~$\xi$ is a problem symmetry and~$x \in \feasregion$, also~$\xi(x) \in
\feasregion$.
Moreover, as all branching variables are represented in~$\sigma_\beta$
and~$x$ respects the branching decisions, $\sigma_\beta(x) =
\sigma_\beta(\xi(x))$ implies that~$\xi(x)$ satisfies the branching
decisions.
Thus, \ref{ex:vardynamic:pr:1} and~\ref{ex:vardynamic:pr:2} hold.
Finally, the \SHCs~\eqref{eq:main} for~$\beta$ dominate the \SHCs
for its ancestors~$\alpha$ since~\eqref{cond:ffunc}
and~$\varphi_\alpha = \id$ hold, i.e, if~$\xi(x)$ satisfies the \SHCs
for~$\beta$, then also all previous \SHCs.
As~$\Gamma \circ \xi = \Gamma$, each~$\gamma \in \Gamma$ can be written
as~$\gamma' \circ \xi$ for some~$\gamma' \in \Gamma$.
Therefore, for all~$\gamma' \in \Gamma$, we conclude
$\sigma_\beta(\xi(x)) = \sigma_\beta(x) \succeq \sigma_\beta(\gamma(x)) =
\sigma_\beta(\gamma'(\xi(x)))$, i.e.,
\ref{ex:vardynamic:pr:3} and thus~\eqref{cond:permutationcondition} holds.
\end{proof}

By adapting the variable order used in \lexfix to the order imposed
by~$\sigma_\beta$, \lexfix is thus compatible with isomorphism pruning and
\orbitalfixing, i.e., the framework achieves goal~\eqref{G1}.
In particular, the statement is true for non-binary problems if these
methods can be generalized to arbitrary variable domains.
We will discuss this in more detail in the next section.

The last symmetry prehandling structure accommodates orbitopal fixing.
Bendotti \etal~\cite{BendottiEtAl2021} already discussed a dynamic variant
of orbitopal fixing, which reorders the rows of the orbitope matrix similar
to Ostrowski's rank; columns, however, are not reordered.
As described above, allowing also column reorderings might lead to more
balanced branch-and-bound trees, which can be achieved as follows.

\begin{example}[Specialized for orbitopal fixing]
  \label{ex:orbitopalfixing}
  Let $M$ be the $p \times q$ orbitope matrix corresponding to the problem
  variables via $M_{i,j} = x_{q(i-1) + j}$.
  That is, $x$ is filled row-wise with the entries of~$M$.
  Let~$\beta \in \bnbnodes$.
  If~$\beta$ is the root node, define $(m_\beta, \pi_\beta, \varphi_\beta) 
  =
  (0, \id, \id)$.
  Otherwise, let~$\alpha$ be the parent of~$\beta$.
  If~$\beta$ arises from~$\alpha$ by a proper branching decision on
  variable~$M_{\ihat,\hat{\jmath}}$ and no variable in the~$\ihat$-th row 
  has been
  used for branching  before,
  set $m_\beta = m_\alpha + q$, select~$\pi_\beta \in \symmetricgroup{n}$ 
  with
  $\pi_\beta(k) = \pi_\alpha(k)$ for~$k \in [m_\alpha]$, 
  and, for~$k \in [q]$, define $\pi_\beta(m_\alpha + k) = q(\ihat-1) + k$.
  Also choose 
  $\psi_\alpha \in \stab{\Gamma}{\feasregion(\alpha)}$
  yielding $\varphi_\beta = \varphi_\alpha\circ \psi_\alpha$.
  Consistent with Condition~\eqref{cond:sibl}, 
  the choice of $\psi_\alpha$ is the same for all children
  sharing the same parent $\alpha$.
  If the variable is already included in the variable ordering
  or if the branching decision is improper, 
  inherit~$(m_\beta, \pi_\beta, \varphi_\beta) 
  = (m_\alpha, \pi_\alpha, \varphi_\alpha)$.
  Effectively, this creates a new matrix in which the rows are sorted based
  on branching decisions and columns can be permuted as long as this does
  not affect symmetrically feasible solutions.  
\end{example}

Completely handling \SHCs~\eqref{eq:main} on $\beta$ corresponds to 
using orbitopal fixing on the $(m_\beta / q) \times q$-matrix 
filled row-wise with the variables 
with indices in~$(\pi_\beta \varphi_\beta)^{-1}(i)$ 
for~$i \in [m_\beta]$.
Bendotti \etal~\cite{BendottiEtAl2021} introduce this 
without the freedom of permuting the matrix columns, 
i.e., for all $\beta \in \bnbnodes$ they choose~$\varphi_\beta = \id$.
We call their setting \emph{row-dynamic}, wheres we refer to our setting as
row- and column-dynamic.

\begin{proof}[Example~\ref{ex:orbitopalfixing} satisfies 
  \eqreffromto{cond:ffunc}{cond:permutationcondition}]
  Obviously, \eqreffromto{cond:ffunc}{cond:sibl} hold.
  To show~\eqref{cond:permutationcondition}, we use induction.
  As~\eqref{cond:permutationcondition} holds at the root node, the 
  induction base holds.
  So, assume~\eqref{cond:permutationcondition} holds at node $\alpha$ 
  with child~$\beta$~{\itshape (IH)}.
  We show \eqref{cond:permutationcondition} also holds at $\beta$.

  Let $x \in \feasregion(\beta)$ and $\xi \in \Gamma$
  with $\sigma_\beta(x) = \sigma_\beta(\xi(x))$.
  To show~$\xi(x) \in \feasregion(\beta)$, we distinguish if the branching
  decision from $\alpha$ to $\beta$ is proper or not.
  Note that both proper and improper branching decisions only happen on
  variables present in~$\sigma_\beta$ by construction
  of~$(m_\beta,\pi_\beta,\varphi_\beta)$.
  Hence, if~${\xi(x) \in \feasregion(\alpha)}$ holds,
  $\sigma_\beta(x) = \sigma_\beta(\xi(x))$
  implies $\xi(x) \in \feasregion(\beta)$.
  Thus, it suffices to prove~${\xi(x) \in \feasregion(\alpha)}$.

  For improper branching decisions,
  $\sigma_\beta = \sigma_\alpha$ and \SHCs~\eqref{eq:main} are propagated.
  As ${x \in \feasregion(\beta) \subseteq \feasregion(\alpha)}$
  and $\sigma_\alpha(x) = \sigma_\alpha(\xi(x))$, 
  {\itshape (IH)} yields $\xi(x) \in \feasregion(\alpha)$.
  For proper branching decisions,
  we observe that
  $\sigma_\beta(\cdot) = 
  \restrict{(\pi_\beta \varphi_\alpha \psi_\alpha(\cdot))}{[m_\beta]}$,
  $\sigma_\alpha(\cdot) = 
  \restrict{(\pi_\beta \varphi_\alpha(\cdot))}{[m_\alpha]}$
  and~$m_\beta \geq m_\alpha$.
  Thus,
  $\sigma_\beta(x) = \sigma_\beta(\xi(x))$
  implies 
  $\sigma_\alpha(\psi_\alpha(x)) = \sigma_\alpha(\psi_\alpha\xi(x))$.
  As~$x \in \feasregion(\beta) \subseteq \feasregion(\alpha)$
  and $\psi_\alpha \in \stab{\Gamma}{\feasregion(\alpha)}$,
  we find~$\psi_\alpha(x) \in \feasregion(\alpha)$.
  By~{\itshape (IH)}, 
  $\sigma_\alpha(\psi_\alpha(x)) = \sigma_\alpha(\psi_\alpha\xi(x))$
  yields $\psi_\alpha \xi(x) \in \feasregion(\alpha)$.
  Again, since $\psi_\alpha \in \stab{\Gamma}{\feasregion(\alpha)}$,
  by applying $\psi_\alpha^{-1}$ left 
  we find $\xi(x) \in \feasregion(\alpha)$.
\end{proof}

\paragraph{Proof of Theorem~\ref{thm:main}}
The examples illustrate that many symmetry prehandling structures are
compatible with the correctness conditions, which shows that there are
potentially many variants to handle symmetries based on
Theorem~\ref{thm:main}.
We proceed to prove this theorem.
To this end, we make use of the following lemma.
\begin{lemma}
\label{lem:gen:main}
\begin{subequations}
  Let~$\feasregion \subseteq \R^n$ and let~$f\colon \feasregion \to \R$ be
  such that~$\optval(f,\feasregion)$ can be solved by (spatial)
  branch-and-bound.
  Let~$\Gamma$ be a finite group of symmetries of~$\optval(f,\feasregion)$.
  Suppose that the branch-and-bound method used for
  solving~$\optval(f,\feasregion)$ generates a full
  \bb-tree~$\bnbtree = (\bnbnodes, \bnbedges)$.
  Let $\beta \in \bnbnodes$ be not a leaf of the \bb-tree.
  If there is a feasible solution $\tilde x \in \feasregion(\beta)$ with
  \begin{equation}
    \label{eq:sigmalexmax}
    \sigma_\beta(\tilde x) \succeq \sigma_\beta \gamma (\tilde x)\
    \text{for all}\ \gamma \in \Gamma,
  \end{equation}
  then $\beta$ has exactly one child $\omega \in \chi_\beta$ for which
  there is~$\xi \in \Gamma$ such that
  \begin{gather}
    \label{eq:xiinfeas}
    \xi(\tilde x) \in \feasregion(\omega),\
    \\
    \text{and}\
    \label{eq:sigmaxilexmax}
    \sigma_\omega \xi (\tilde x) \succeq \sigma_\omega \gamma \xi (\tilde x)\
    \text{for all}\ \gamma \in \Gamma.
  \end{gather}
\end{subequations}
\end{lemma}
\begin{proof}
Let $\tilde x \in \feasregion(\beta)$ respect~\eqref{eq:sigmalexmax}.
First, we show the existence of $\omega \in \chi_\beta$
satisfying~\eqref{eq:xiinfeas} and~\eqref{eq:sigmaxilexmax}.
Thereafter, we show that $\omega$ is unique.

\vspace{1em}
\noindent
\emph{Existence:}\quad
By Condition~\eqref{cond:sibl},
the maps $\sigma_\omega$ for all children~$\omega \in \chi_\beta$
are the same.
Let~$\xi \in \Gamma$ be such that~$\sigma_\omega \xi(\tilde x)$ is
lexicographically maximal.
Note that~$\xi$ exists, since $\Gamma$ is a finite group.
Then, $\sigma_\omega \xi(\tilde x) \succeq \sigma_\omega \gamma \xi(\tilde
x)$, because~$\xi, \gamma \in \Gamma$ implies~$\xi\gamma \in \Gamma$.
Thus, $\xi$ satisfies~\eqref{eq:sigmaxilexmax}.
We show that~$\xi(\tilde x) \in \feasregion(\omega)$
for some $\omega \in \chi_\beta$.

Recall that the branching decision at $\beta$ partitions its feasible
region, i.e., $\{ \feasregion(\omega) : \omega \in \chi_\beta \}$ partitions
$\feasregion(\beta)$.
As such, there is exactly one child $\omega \in \chi_\beta$
with $\xi(\tilde x) \in \feasregion(\omega)$
if $\xi(\tilde x) \in \feasregion(\beta)$.
To show \eqref{eq:xiinfeas}, it thus suffices
to prove~$\xi(\tilde x) \in \feasregion(\beta)$.

For any child $\omega \in \chi_\beta$, 
vector~$x$, and~$i \leq m_\beta$, we have
\begin{equation}
\label{eq:sigmabetasubstitutesigmaomega}
(\sigma_\omega(x))_i
=
(\pi_\omega \varphi_\omega (x))_i
\stackrel{\eqref{cond:ffunc}}=
(\pi_\beta \varphi_\omega (x))_i
\stackrel{\eqref{cond:symmetricallypermute}}=
(\pi_\beta \varphi_\beta \psi_\beta (x))_i
=
(\sigma_\beta \psi_\beta (x))_i
.
\end{equation}
Recall that $\xi \in \Gamma$ satisfies~\eqref{eq:sigmaxilexmax}.
Substituting \eqref{eq:sigmabetasubstitutesigmaomega}
yields
$\sigma_\beta \psi_\beta \xi (\tilde x)
\succeq
\sigma_\beta \psi_\beta \gamma \xi (\tilde x)$
for all $\gamma \in \Gamma$.
In particular, for $\gamma = \psi_\beta^{-1}\xi^{-1} \in \Gamma$,
we find
$\sigma_\beta \psi_\beta \xi (\tilde x)
\succeq \sigma_\beta(\tilde x)$.
Then~\eqref{eq:sigmalexmax}
yields $\sigma_\beta \psi_\beta \xi (\tilde x)
= \sigma_\beta(\tilde x)$.
By~\eqref{cond:permutationcondition},
we thus have $\psi_\beta \xi(\tilde x) \in \feasregion(\beta)$.
Since $\psi_\beta \in \stab{\Gamma}{\feasregion(\beta)}$,
applying $\psi_\beta^{-1}$ left on this solution yields
$\xi(\tilde x) \in \feasregion(\beta)$,
herewith completing the first part.

\vspace{1em}
\noindent
\emph{Uniqueness:}\quad
Suppose $\xi, \xi' \in \Gamma$ satisfy~\eqref{eq:sigmaxilexmax}.
For~$\gamma = \xi' \xi^{-1}$, \eqref{eq:sigmaxilexmax} for~$\xi$ implies
$\sigma_\omega \xi(\tilde x) \succeq \sigma_\omega \xi'(\tilde x)$.
Analogously, for $\xi'$ we choose $\gamma = \xi (\xi')^{-1}$ to find
$\sigma_\omega \xi'(\tilde x) \succeq \sigma_\omega \xi(\tilde x)$.
As a result,
$\sigma_\omega \xi(\tilde x) = \sigma_\omega \xi'(\tilde x)$.

Suppose $\xi(\tilde x) \in \feasregion(\omega)$.
Let $x = \xi(\tilde x)$ and $\gamma = \xi' \xi^{-1} \in \Gamma$.
Then, we find
\[
  \sigma_\omega(x)
  =
  \sigma_\omega \xi(\tilde x)
  =
  \sigma_\omega \xi'(\tilde x)
  =
  \sigma_\omega \xi'(\xi^{-1}(x))
  =
  \sigma_\omega \gamma(x),
\]
and Condition~\eqref{cond:permutationcondition}
yields $\xi'(\tilde x) = \gamma(x) \in \feasregion(\omega)$.
As the children $\chi_\beta$ partition $\feasregion(\beta)$
and $\xi'(\tilde x) \in \feasregion(\omega)$,
there is no other child of $\beta$ where $\xi'(\tilde x)$ is
feasible.
Thus, independent from $\xi \in \Gamma$
satisfying~\eqref{eq:sigmaxilexmax},
there is exactly one child~$\omega \in \chi_\beta$
with~$\xi(\tilde x) \in \feasregion(\omega)$.
\end{proof}
We are now able to prove Theorem~\ref{thm:main}.
\begin{proof}[Proof of Theorem~\ref{thm:main}]
  Recall that we assumed~$\bnbtree = (\bnbnodes, \bnbedges)$ to be finite
  and that we do not prune nodes by bound.
  Let~$\bnbtree_d$ be the tree arising from~$\bnbtree$ by pruning all nodes
  at depth larger than~$d$.
  Let~$(m_\beta,\pi_\beta,\varphi_\beta)_{\beta \in \bnbnodes}$ satisfy the
  correctness conditions.
  Let~$\check{x} \in \feasregion$  be any feasible solution to the original
  problem.
  We proceed by induction and show that, for every depth~$d$ of the tree,
  there is exactly one leaf node in~$\bnbtree_d$ for which a permutation
  of~$\check{x}$ is feasible and that does not violate the local
  \SHCs~\eqref{eq:main}.

  Let~$d = 0$.
  The only node at depth~$d$ is the root node~$\alpha$.
  Any feasible solution~$\check x \in \feasregion$
  is feasible in the root node~$\alpha \in \mathcal V$
  of the branch-and-bound tree~$\mathcal B$.
  In particular, we can permute~$\check x$ by any $\xi \in \Gamma$,
  and have a feasible symmetrical solution.
  For the root node, choose $\xi \in \Gamma$ such that~$\sigma_\alpha
  \xi(\check x) \succeq \sigma_\alpha \gamma \xi(\check x)$
  for all~$\gamma \in \Gamma$.
  That is, $\xi(\check{x})$ is not cut off by~\eqref{eq:main} at~$\alpha$.

  Let~$d > 0$ and let~$\tilde{x} \in \feasregion$.
  By induction, we may assume that there is exactly one leaf node~$\beta$
  of~$\bnbtree_d$ at which a permutation~$\xi(\check{x})$ is feasible and
  that is not cut off by~\eqref{eq:main}.
  If~$\beta$ is also a leaf in~$\bnbtree$, we are done.
  Otherwise, since~$\xi(\check{x})$ is not cut off by~\eqref{eq:main}, we
  can apply Lemma~\ref{lem:gen:main} and find that~$\beta$ has exactly one
  child~$\omega$ at which a permutation of~$\xi(\check{x})$ is feasible and
  is not cut off by~\eqref{eq:main} at node~$\omega$.
  This concludes the proof.
\end{proof}
\begin{remark}
  For spatial branch-and-bound algorithms, two subtleties arise.
  On the one hand, there might not exist a finite branch-and-bound tree.
  If all branching decisions partition a subproblem's feasible region,
  Theorem~\ref{thm:main} holds true for all trees pruned at a certain depth
  level.
  On the other hand, branching decisions do not necessarily partition the
  feasible region.
  In this case, \eqref{eq:main} can still be used to handle symmetries.
  However, in the depth-pruned tree there might exist more than one leaf
  containing a symmetric copy of a feasible solution.
\end{remark}
\begin{remark}
  Theorem~\ref{thm:main} still holds in case of some infinite groups.
  The only place where finiteness is used is in the proof of
  Lemma~\ref{lem:gen:main}, where it implies that a symmetry~$\xi \in \Gamma$
  exists such that $\sigma_{\omega}\xi(\tilde x)$ is lexicographically maximal
  for a fixed solution vector $\tilde x \in \feasregion(\beta)$.
  For instance, for infinite groups of rotational symmetries, such a
  symmetry always exists.
\end{remark}
\section{Apply framework on generic optimization problems}
\label{sec:cast}

Due to Theorem~\ref{thm:main}, we can completely handle all
symmetries of an arbitrary problem~$\optval(f, \feasregion)$, provided we
know how to handle Constraints~\eqref{eq:main}.
The aim of this section is therefore to find symmetry handling methods that
can deal with non-binary variables.
Since handling Constraints~\eqref{eq:main} is already difficult for binary
problems, we cannot expect to handle all symmetries efficiently.
Instead, we revisit the efficient methods \lexfix, orbitopal fixing, and
\orbitalfixing for binary variables and provide proper generalizations for
non-binary problems, which allows us to partially enforce
Constraints~\eqref{eq:main}.
We refer to these generalizations as lexicographic reduction, orbitopal
reduction, and orbital symmetry handling, respectively.

Throughout this section, we assume that~$\Gamma \leq \symmetricgroup{n}$.

\subsection{Lexicographic reduction}
\label{sec:gen:lexred}

\subsubsection{The static setting}

Assume the symmetry prehandling structure of Example~\ref{ex:staticsetting}
is used in Theorem~\ref{thm:main}.
Then, the \SHCs~$x \succeq \gamma(x)$ for all~$\gamma \in \Gamma$ are
enforced at each node of the branch-and-bound tree.
For all~$i \in [n]$, let~$\domain_i \subseteq \R^n$ be the domain of
variable~$x_i$ at a node of the branch-and-bound tree and let~$\domain =
(\domain_i)_{i \in [n]}$ be the vector of variable domains.
The aim of the lexicographic reduction (\lexred) algorithm is to find, for
a fixed permutation~$\gamma \in \Gamma$, the smallest domains~$\domain'_i$,
$i \in [n]$, such that
$\left\{ x \in \bigtimes_{i = 1}^n \domain_i : x \succeq \gamma(x)\right\} =
\left\{ x \in \bigtimes_{i = 1}^n \domain'_i : x \succeq \gamma(x)\right\}$.

If~$\domain_i \subseteq \binary$ for all~$i \in [n]$, the reductions found
by \lexred are equivalent to the reductions found by \lexfix.
For non-binary domains, similar ideas as for \lexfix, which are
described
in~\cite{BestuzhevaEtal2021OO,doornmalenhojny2022cyclicsymmetries}, can be used:
We iterate over the variables~$x_i$ with indices in increasing order.
If~$x_j = \gamma(x)_j$ for all indices~$j < i$, we enforce~$x_i \geq \gamma(x)_i$,
and we check if a solution with~$x_i = \gamma(x)_i$ exists.
Before we provide a rigorous algorithm, we illustrate the idea.

\begin{example}
  \label{ex:lexred}
Let $\feasregion = [-1, 1]^4 \cap \Z^4$ and $\gamma = (1,3,2,4)$.
Consider a node with relaxed region 
$x \in \{ 0 \} \times [-1, 0] \times \{ 1 \} \times [-1, 1]$.
Propagating 
$x \succeq \gamma(x)$,
we find
{\footnotesize
\begin{equation*}
\begin{bmatrix}
x_1& =&   0\\ 
x_2& \in& [-1,0]\\ 
x_3& =&   1 \\ 
x_4& \in& [-1, 1]\\
\end{bmatrix}
\succeq
\begin{bmatrix}
x_4& \in& [-1, 1]\\
x_3& =&   1 \\ 
x_1& =&   0\\ 
x_2& \in& [-1,0]\\ 
\end{bmatrix}
\stackrel{\text{(}\dagger\text{)}}{\leadsto}
\begin{bmatrix}
x_1& =&   0\\ 
x_2& \in& [-1,0]\\ 
x_3& =&   1 \\ 
x_4& \in& [-1, 0]\\
\end{bmatrix}
\succeq
\begin{bmatrix}
x_4& \in& [-1, 0]\\
x_3& =&   1 \\ 
x_1& =&   0\\ 
x_2& \in& [-1,0]\\ 
\end{bmatrix}
\!.
\end{equation*}%
}%
In ($\dagger$), we restrict the domain of $x_4$ by propagating 
$0 = x_1 \geq x_4$, resulting in $x_4 \in [-1, 0]$.
If~${x_1 = x_4 = 0}$,
then \SHC $x \succeq \gamma(x)$ implies 
the contradiction ${[-1, 0] \ni x_2 \geq x_3 = 1}$,
so we must have $x_1 > x_4$. Since $x_4 \in \Z$,
$x_4$ must be fixed to $-1$.
No further domain reductions can be derived from $x \succeq \gamma(x)$.
\end{example}

We now proceed with our generalization of \lexfix.
To enforce~$x \succeq \gamma(x)$ for general variable domains~$\domain$,
some artifacts need to be taken into account.
For example, if~$n = 3$ and~$\gamma$ is the cyclic right-shift,
then~$y^\epsilon \define(1+\epsilon,0,1) \succeq \gamma(y^\epsilon) =
(1,1+\epsilon,0)$ for every~$\epsilon > 0$, but~$y^0 \prec \gamma(y^0)$,
i.e., $\{x \in \R^n : x \succeq \gamma(x)\}$ is not necessarily closed.
Since optimization software usually can only handle closed sets, we propose
the following solution.
We extend~$\R$ by an infinitesimal symbol~$\varepsilon$ that we can add to
or subtract from any real number to represent a strict difference.
This results in a symbolically correct algorithm
that is as strong as possible.
For example, $\min\{ 1 + x : x > 1 \} = 2 + \varepsilon$,
$\min\{1 + x + \varepsilon : x > 1 \} = 2 + \varepsilon$,
$\max\{ 1 + x : x < 2 \} = 3 - \varepsilon$, 
and we do not allow further arithmetic 
with the $\varepsilon$ symbol.
In practice, however, we cannot enforce strict inequalities.
We thus replace~$\varepsilon$ by~0, which will lead to slightly
weaker but still correct reductions.
That is no problem for our purposes,
since we will only either apply the $\min$-operator or the 
$\max$-operator, the sign of $\varepsilon$ will always be the same;
namely, if $\varepsilon$ appears, 
this has a positive sign in minimization-operations,
and a negative sign in maximization-operations.

Now, we turn to the generalization of \lexfix to arbitrary variable domains.
We introduce a timestamp $t$.
At every time $t$, the current domain is denoted by~$\domain^t$.
We initialize~$\domain_i^0 = \domain_i$ for all~$i \in [n]$, and for two
timestamps~$t > t'$, we will possibly strengthen the domains, i.e.,
$\domain_i^{t} \subseteq \domain_i^{t'}$.

The core of \lexred is the observation that if~$\restrict{x}{[t-1]} =
\restrict{\gamma(x)}{[t-1]}$ holds for some~$t \geq 1$, then constraint~$x \succeq
\gamma(x)$ can only hold when~$x_t \geq \gamma(x)_t = x_{\gamma^{-1}(t)}$.
This observation is exploited in a two-stage approach.
In the first stage, \lexred performs the following steps for all~$t =
1,\dots,n$:
\begin{enumerate}
\item The algorithm propagates~$x_t \geq \gamma(x)_t$ by
  updating the variable domains via
\begin{equation}
\label{eq:VDR}
\begin{aligned}
\domain_t^t 
&= \{ z \in \domain_t^{t-1} :
z \geq \min(\domain_{\gamma^{-1}(t)}^{t-1}) \}
,\\
\domain_{\gamma^{-1}(t)}^t 
&= \{ z \in \domain_{\gamma^{-1}(t)}^{t-1} :
z \leq \max(\domain_{t}^{t-1}) \}
, \text{ and}\\
\domain_i^t
&= \domain_i^{t-1}\ \text{for}\ i \in [n] \setminus \{ t, \gamma^{-1}(t) \}.
\end{aligned}
\end{equation}
\item Then, it checks whether~$\domain^t_i \neq \emptyset$ for all~$i \in
  [n]$ and whether~$x \in \domain^t$ guarantees~$\restrict{x}{[t]} =
  \restrict{\gamma(x)}{[t]}$.
  If this is the case, the algorithm continues with iteration~$t+1$.
  Otherwise, the first phase of \lexred terminates, say at time~$t^\star$.
\end{enumerate}
Of course, all variable domain reductions found during phase one are
correct based on the previously mentioned observation.

At the end of phase one, three possible cases can occur: a variable domain
is empty, phase one has propagated all variables, i.e., $x = \gamma(x)$ for
all~$x \in \domain^n$, or~$\restrict{x}{[t^\star-1]} =
\restrict{\gamma(x)}{[t^\star-1]}$ and there exists~$(v,w) \in
\domain^{t^\star}_{t^\star} \times
\domain^{t^\star}_{\gamma^{-1}(t^\star)}$ with~$v \neq w$.
In either of the first two cases, the algorithm stops because it either
has shown that no solution~$x \in \domain^0$ exists with~$x \succeq
\gamma(x)$ or all variables are fixed.
In the last case, note that~$v > w$ holds due to the domain reductions
at time~$t^\star$.
Since~$\restrict{x}{[t^\star-1]} = \restrict{\gamma(x)}{[t^\star-1]}$ holds
for all~$x \in \domain^{t^\star}$, the relation~$v > w$ shows that there
exists~$x \in \domain^{t^\star}$ such that~$\restrict{x}{[t^\star]} \succ
\restrict{\gamma(x)}{[t^\star]}$.
Consequently, the domains of variables~$x_{t^\star+1},\dots,x_n$ cannot
be tightened.
It might be possible, however, that the domains of~$x_{t^\star}$
and~$\gamma(x)_{t^\star}$ can be reduced further.
Namely, if~$x_{t^\star} = \min \domain^{t^\star}_{\gamma^{-1}(t^\star)}$
or~$x_{\gamma^{-1}(t^\star)} = \max \domain^{t^\star}_{t^\star}$.
In this case, the other variable necessarily attains the same value, which
means that a solution with~$\restrict{x}{[t^\star]} =
\restrict{\gamma(x)}{[t^\star]}$ is created, which might lead to a
contradiction with~$x \succeq \gamma(x)$ as illustrated in
Example~\ref{ex:lexred}.

In the second stage of \lexred, it is checked whether one of these cases
indeed leads to a contradiction.
If this is the case, $\min \domain^{t^\star}_{\gamma^{-1}(t^\star)}$ can be
removed from the domain of~$x_{t^\star}$ or~$\max
\domain^{t^\star}_{t^\star}$ can be removed from the domain of~$x_{\gamma^{-1}(t)}$.
To detect whether a contradiction occurs, the second phase hypothetically
fixes~$x_{t^\star}$ or~$x_{\gamma^{-1}(t^\star)}$ to the respective value
and continues with stage one since~$\restrict{x}{[t^\star]} =
\restrict{\gamma(x)}{[t^\star]}$ now holds.
If phase one then terminates because a variable domain becomes empty, this
shows that the domain of~$x_{t^\star}$ or~$x_{\gamma^{-1}(t^\star)}$ can be
reduced.
Otherwise, no further variable domain reductions can be derived.
\begin{proposition}
  Let~$\tau$ be the time needed to perform one variable domain reduction
  in~\eqref{eq:VDR}.
  Then, \lexred finds all possible variable domain reductions for~$x
  \succeq \gamma(x)$ in~$\bigo(n \cdot \tau)$ time.
\end{proposition}
\begin{proof}
  Completeness of \lexred follows from the previous discussion.
  The running time holds as the first stage computes at most~$n$
  domain  reductions and the second stage triggers phase
  one at most twice.
\end{proof}
In many cases, for instance, if variable domains are continuous or discrete
intervals, $\tau = \bigo(1)$, turning lexicographic reduction into a linear
time algorithm.
\subsubsection{Dynamic settings}

Theorem~\ref{thm:main}
shows that $\sigma_\beta(x) \succeq \sigma_\beta \gamma(x)$
is a valid symmetry handling constraint
for certain symmetry prehandling 
structures~$(m_\beta,\pi_\beta,\varphi_\beta)_{\beta \in \bnbnodes}$.
If~$\Gamma$ is a permutation group,
$\sigma_\beta(x)$ and $\sigma_\beta(\gamma(x))$
are just permutations of the solution vector entries 
and a restriction of this vector.
In this case, lexicographic reduction can, of course,
also propagate these \SHCs by changing the order in which 
we iterate over the solution vector entries.

In particular, in the binary case and the symmetry prehandling structure of
Example~\ref{ex:vardynamic}, the adapted version of \lexred is compatible with
\isomorphismpruning and \orbitalfixing as we have seen in
Section~\ref{sec:framework} that the latter two methods
propagate~$\sigma_\beta(x) \succeq \sigma_\beta \gamma(x)$ for
all~$\gamma\in\Gamma$.

\subsection{Orbitopal reduction}
\label{sec:gen:orbitopalfixing}
Bendotti \etal~\cite{BendottiEtAl2021} present a complete propagation 
algorithm to handle orbitopal symmetries on binary variables.
In this section, we generalize their algorithm to arbitrary variable
domains.
We call the generalization of orbitopal fixing \emph{orbitopal reduction}
as it does not necessarily fix variables.

\subsubsection{The static setting}

Suppose that~$\Gamma$ is the group that contains all column permutations of
a $p \times q$ variable matrix~$X$.
Further, assume that Theorem~\ref{thm:main} uses the symmetry prehandling structure
from Example~\ref{ex:staticsetting}, where we assume that the variable
vector~$x$ associated with the~$p \cdot q$ variables in~$X$ is
such that enforcing~$x \succeq \gamma(x)$ for all~$\gamma \in \Gamma$
corresponds to sorting the columns of~$X$ in lexicographic order.
With slight abuse of notation, for $\gamma \in \Gamma$, we
write $X \succeq \gamma(X)$ if and only if the corresponding vector~$x \in
\R^{pq}$ satisfies~$x \succeq \gamma(x)$.

We use the following notation.
For any~$M \in \R^{p \times q}$ and~$(i,j) \in [p] \times [q]$, we denote
by $M_i$ the $i$-th row of~$M$, by~$M^j$ the $j$-th column of~$M$, and
by~$M_i^j$ the entry at position $(i, j)$.
For every variable~$X_i^j$, we denote its domain by 
$\domain_i^j \subseteq \R$. 
Using the same matrix notation, 
$\domain \subseteq \R^{p \times q}$
denotes the $p \times q$-matrix where entry $(i, j)$ corresponds 
to~$\domain_i^j$.
For given domain~$\domain \subseteq \R^{p \times q}$, we denote
by~$\underline{M}(\domain)$ and~$\overline{M}(\domain)$ the
lexicographically smallest and largest element in~$\domain$, respectively.
Whenever the domain~$\domain$ is clear from the context, we just
write~$\underline{M}$ and~$\overline{M}$.
Moreover, let
\[
  O_{p \times q} \define \{ X \in \R^{p \times q} : X \succeq \gamma(X)\ \text{for all}\ \gamma \in \Gamma \}
\]
be the set of all matrices with lexicographically sorted columns.
Our goal is to find all possible \VDRs
of the \SHCs $X \succeq \gamma(X)$ for $\gamma \in \Gamma$, i.e.,
we want to find the smallest~$\hat{\domain} \subseteq \R^{p \times q}$ such
that
\[
  \hat{\domain} \cap O_{p \times q} = \domain \cap O_{p \times q}.
\]
It turns out that, as for the binary case~\cite{BendottiEtAl2021}, the
matrices~$\underline{M}(\domain)$ and~$\overline{M}(\domain)$ contain
sufficient information for finding~$\hat{\domain}$.
In the following, recall that we (implicitly) use the infinitesimal
notation introduced in the previous section to represent strict
inequalities.
\begin{theorem}
  \label{thm:genorbitopalfixing}
  Let~$\domain \subseteq \R^{p \times q}$.
  For~$j \in [q]$, let~$i_j \define \min(\{ i \in [p] :
  \underline{M}(\domain)_i^j \neq \overline{M}(\domain)_i^j \} \cup \{ p +
  1 \})$.
  Then, the smallest~$\hat{\domain} \subseteq \R^{p \times q}$ for
  which~$\hat{\domain} \cap O_{p \times q} = \domain \cap O_{p \times q}$
  holds satisfies, for every~$(i,j) \in [p] \times [q]$,
  \[
    \hat{\domain}^j_i
    =
    \begin{cases}
      \domain^j_i \cap [\underline{M}(\domain)_i^j,
      \overline{M}(\domain)_i^j],
      & \text{if } i \leq i_j,\\
      \domain^j_i, &\text{otherwise}.
    \end{cases}
  \]
\end{theorem}
This theorem is proven by the following two lemmas.
The first lemma shows that no tighter \VDRs can be achieved:
for every $(i, j) \in [p] \times [q]$
and $v \in \hat\domain_i^j$ a lexicographically non-increasing solution 
matrix $\tilde X$ exists with $\tilde X_i^j = v$.
The second lemma shows that the \VDRs are valid:
for every~$(i, j) \in [p] \times [q]$
and $v \in \domain_i^j \setminus \hat\domain_i^j$,
no matrix $\tilde X$ with $\tilde X_i^j = v$ exists.
\begin{lemma}
\label{lem:genorbitopalfixing:tight}
Suppose that $O_{p \times q} \cap \domain \neq \emptyset$.
Let $i' \in [p]$ and $j' \in [q]$ with $i' \leq i_{j'}$.
For all $x \in \domain_{i'}^{j'}$
with~$\underline M _{i'}^{j'} \leq x \leq \overline M _{i'}^{j'}$
there is $X \in O_{p \times q} \cap \domain$
with $X_{i'}^{j'} = x$.
\end{lemma}
\begin{proof}
We define two matrices $A, B \in \domain$, for which entries~$(i,j) \in [p]
\times [q]$ are
\[
A_i^j =
\begin{cases}
\overline M_i^j & \text{if}\ j < j', \\
\underline M_i^j & \text{if}\ j > j', \\
\overline M_i^j = \underline M_i^j & \text{if}\ j = j', i < i_j, \\
\overline M_i^j (> \underline M_i^j) & \text{if}\ j = j', i = i_j,\\
\min(\domain_i^j) & \text{if}\ j = j', i > i_j,
\end{cases}
\ 
\text{and}\ 
B_i^j =
\begin{cases}
\overline M_i^j & \text{if}\ j < j', \\
\underline M_i^j & \text{if}\ j > j', \\
\overline M_i^j = \underline M_i^j & \text{if}\ j = j', i < i_j, \\
\underline M_i^j (< \overline M_i^j) & \text{if}\ j = j', i = i_j, \\
\max(\domain_i^j) & \text{if}\ j = j', i > i_j.
\end{cases}
\]
From these two matrices, we show that for any~$x \in \domain_{i'}^{j'}$
with~$\underline M _{i'}^{j'} \leq x \leq \overline M _{i'}^{j'}$ there
is~$X \in O_{p \times q} \cap \domain$ with $X_{i'}^{j'} = x$.
We call such a matrix~$X$ a certificate for~$x$.
In the following, we first provide a construction for these certificates,
and after that we show that they are contained 
in~$\domain \cap O_{p \times q}$.

If $i' < i_{j'}$, then $x = \overline M_{i'}^{j'} = \underline M_{i'}^{j'}$.
Thus, $X = A$ is a certificate.
If $i' = i_{j'}$,
then there are three options:
If $x = \underline M_{i'}^{j'}$,
then $X = B$ is a certificate;
if $x = \overline M_{i'}^{j'}$,
then $X = A$ is a certificate;
and if~$\underline M_{i'}^{j'} < x < \overline M_{i'}^{j'}$,
then construct $X$ with
$X_i^j = A_i^j$ if $(i, j) \neq (i', j')$,
and $X_{i'}^{j'} = x$.

Note that $X \in \domain$.
We finally show that $X \in O_{p \times q}$, concluding the proof.
The first $j' - 1$ columns of $X$ correspond to $\overline M$.
That is, they satisfy $X^j \succeq X^{j + 1}$ for all $1 \leq j < j' - 1$.
Similarly, the columns after column $j'$ correspond to~$\underline M$.
Hence, $X^j \succeq X^{j + 1}$ for all $j' < j < q$.
By the definition of $A$ and $B$,
$\overline M^{j' - 1} \succeq \overline M^{j'} \succeq A^{j'}$ 
and 
$B^{j'} \succeq \underline M^{j'} \succeq \underline M^{j' + 1}$.
As the columns of~$X$ are either columns of~$A$ or~$B$, or equal
to~$A^{j'}$ up to one entry while remaining lexicographically larger than $B^{j'}$,
we find~$\overline M^{j' - 1} = A^{j' - 1} = X^{j' - 1} 
\succeq A^{j'} \succeq X^{j'} \succeq B^{j'} \succeq B^{j' + 1} = X^{j' +
  1}  = \underline M^{j' + 1}$.
So, for all consecutive $j\in [q - 1]$, we have~$X^j \succeq X^{j + 1}$,
and hence $X \in O_{p \times q} \cap \domain$.
\end{proof}
\begin{lemma}
\label{lem:genorbitopalfixing:valid}
Suppose that $O_{p \times q} \cap \domain \neq \emptyset$.
Let $i' \in [p]$ and $j' \in [q]$ with $i' \leq i_{j'}$.
For all $X \in O_{p \times q} \cap \domain$, we have
$\underline M _{i'}^{j'} \leq X_{i'}^{j'} \leq \overline M _{i'}^{j'}$.
\end{lemma}
\begin{proof}
Suppose the contrary, i.e.,
for $X \in O_{p \times q} \cap \domain$
either $X_{i'}^{j'} < \underline M _{i'}^{j'}$
or $X_{i'}^{j'} > \overline M _{i'}^{j'}$.
Suppose that~$i'$ is minimal, i.e., there is no $i'' < i'$
with $X_{i''}^{j'} < \underline M _{i''}^{j'}$
or $X_{i''}^{j'} > \overline M _{i''}^{j'}$.
By symmetry, it suffices to consider the case
$X_{i'}^{j'} < \underline M _{i'}^{j'}$.

If $X_{i}^{j'} \leq \underline M _{i}^{j'}$ holds for all $i < i'$,
then $X^{j'} \prec \underline M^{j'}$,
which contradicts that $\underline M$ is the lexicographically minimal 
solution of $O_{p \times q} \cap \domain$.
Hence, there is a row $i'' < i'$ 
with $X_{i''}^{j'} > \underline M_{i''}^{j'}$.
Since~$i'' < i' \leq i_{j'}$
yields $\underline M_{i''}^{j'} = \overline M_{i''}^{j'}$,
we have $X_{i''}^{j'} > \overline M_{i''}^{j'}$.
This contradicts that $i'$ is supposed to be minimal 
with~$X_{i'}^{j'} < \underline M_{i'}^{j'}$
or~$X_{i'}^{j'} > \overline M_{i'}^{j'}$,
since for~$i'' < i'$ we satisfy the second condition.
This is a contradiction.
\end{proof}
\begin{proof}[Proof of Theorem~\ref{thm:genorbitopalfixing}]
  Lemmas~\ref{lem:genorbitopalfixing:tight}
  and~\ref{lem:genorbitopalfixing:valid} prove the assertion for $i' \leq
  i_{j'}$.
  Since the domains for~${i' > i_{j'}}$ are not restricted in comparison
  to~$\domain$, domain~$\hat{\domain}^{j'}_{i'}$ is valid.
  To show that it is as tight as possible, we can reconsider in the proof of
  Lemma~\ref{lem:genorbitopalfixing:tight} the matrix $A \in \domain \cap
  O_{p \times q}$.
  Replacing entry~$(i', j')$ in $A$ with any value in $\domain_{i'}^{j'}$
  yields a matrix $\tilde A$.
  If $i' > i_{j'}$, this change does not affect the lexicographic order
  constraint, so $\tilde A \in \domain \cap O_{p \times q}$ is a certificate
  of tightness.
  Combining these statements shows correctness of 
  Theorem~\ref{thm:genorbitopalfixing}.
\end{proof}
We conclude this section with an analysis of the time needed to
find~$\hat{\domain}$.
The crucial step is to find the matrices~$\underline{M}$
and~$\overline{M}$.
To find these matrices, we adapt the idea from~\cite{BendottiEtAl2021} for
the binary case.
Denote by~$\lexmin(\cdot)$ and~$\lexmax(\cdot)$ the operators that determine
the lexicographically minimal and maximal elements of a set, respectively.
We claim that for the lexicographically minimal
element~$\underline{M}$, the~$j$-th column is
\[
\underline M^j = 
\begin{cases}
\lexmin \{ X \in \domain^j : X \succeq \underline M^{j+1} \},
& \text{if}\ j < q, \\
\lexmin \{ X \in \domain^j \},
& \text{otherwise}.
\end{cases}
\]
This can be computed iteratively, starting with the last column $j=q$,
and then iteratively reducing~$j$ until the first column.
The arguments for correctness are the same as
in~\cite[Thm.~1, Lem.~2]{BendottiEtAl2021}.
For this reason, we only describe how to compute the~$j$-th column.
Due to this iterative approach, 
when computing column $\underline{M}^{j}$
for $j > q$, column $\underline{M}^{j+1}$ is known.
The idea is to choose the entries of 
$\underline{M}^{j}$ minimally
such that $\underline{M}^{j} \succeq \underline{M}^{j+1}$ holds.
This resembles the propagation method of the previous section (\lexred),
by choosing the entries minimally such that the constraint holds
when restricted to the first elements, then increasing the 
vector sizes by one iteratively.
If this leads to a contradiction with the constraint, 
it is returned to the last step where the entry
was not fixed, and this entry is increased 
to repair feasibility of the constraint.

More precisely,
$\underline M^{j}$ is found by iterating~$i$ from $1$ to~$p$ 
as follows.
If there is a row index~${i' < i}$ 
with $\underline M_{i'}^j > \underline M_{i'}^{j + 1}$,
let~$\underline M_i^j \gets \min ( \domain_i^j )$.
This is possible, because row~$i'$ already guaranteed that~$\underline{M}^j
\succ \underline{M}^{j+1}$.
If no such index exists, we may assume that all preceding rows $i' < i$
have~${\underline M_i^j = \underline M_i^{j+1}}$ (otherwise, the~$j$-th
column cannot be lexicographically larger than column~$j+1$ as becomes
clear in the following).
In this case, denote~$S^i \define 
\{ x \in \domain_i^j : x \geq \underline M_i^{j+1} \}$.
On the one hand, 
if $\card{S^i} > 0$, set $\underline M_i^j \gets \min(S^i)$.
If this yields $\underline M_i^j > \underline M_i^{j + 1}$,
then stop the iteration, and for all $i'' > i$ 
set~$\underline M_{i''}^{j} \gets \min(\domain_{i''}^j)$.
This makes sure that~$\underline{M}^j$ is lexicographically strictly larger
than~$\underline{M}^{j+1}$.

On the other hand, 
if $\card{S^i} = 0$, we cannot enforce~$\underline{M}^j \succ
\underline{M}^{j+1}$ in row~$i$.
To ensure~$\underline{M}$ becomes the lexicographically smallest
element in~$O_{p \times q} \cap \domain$, we return to the largest $i' < i$
with~$\card{S^{i'}} > 1$ and enforce a lexicographic difference by setting
$\underline M_{i'}^j \gets \min\{ x \in \domain_{i'}^j : 
x > \underline M_{i'}^{j+1} \}$,
and, for all~$i'' > i'$, we assign~$\underline M_{i''}^{j} \gets \min(\domain_{i''}^j)$.
If no $i' < i$ exists with~$\card{S^{i'}} > 1$, column~$j$ cannot become
lexicographically at least as large as column~$j+1$.
That is, $\domain \cap O_{p \times q} = \emptyset$.

Analogously, one computes $\overline M$ by
\[
\overline{M}^j = 
\begin{cases}
\lexmax\{ X \in \domain^j : \overline M^{j - 1} \succeq X \}
& \text{if}\ j > 1,\ \text{and} \\
\lexmax\{ X \in \domain^j \}
& \text{otherwise}.
\end{cases}
\]
Since determining the~$j$-th column of~$\underline{M}$
and~$\overline{M}$ requires to iterate over its elements a
constant number of times,
for each element a constant number of comparisons and variable domain
reductions is executed.
The time for finding~$\underline{M}$ and~$\overline{M}$
is therefore~$\bigo(pq\tau)$, where~$\tau$ is again the time needed to reduce
variable domains.
Combining all arguments thus yields the following result regarding
orbitopal reduction.
\begin{theorem}
  Let~$\domain \subseteq \R^{p \times q}$.
  Orbitopal reduction finds the smallest~$\hat{\domain} \subseteq \R^{p
    \times q}$ such that~$\hat{\domain} \cap O_{p \times q} = \domain \cap
  O_{p \times q}$ holds in~$\bigo(pq\tau)$ time.
  In particular, if all variable domains are intervals, orbitopal reduction
  can be implemented to run in~$\bigo(pq)$ time.
\end{theorem}

\subsubsection{Dynamic settings}

Similar to \lexred, also orbitopal reduction can be used to propagate
\SHCs~$\sigma_\beta(x) \succeq \sigma_\beta\gamma(x)$ for
permutations~$\gamma$ from a group~$\Gamma$ of orbitopal symmetries.
The only requirement is that~$\sigma_\beta$ is compatible with the matrix
interpretation of a solution~$x$, which can be achieved by using the
symmetry prehandling structure of Example~\ref{ex:orbitopalfixing}.
In this case, the static orbitopal reduction algorithm is only applied to
the variables ``seen'' by~$\sigma_\beta(x) \succeq \sigma_\beta\gamma(x)$.

Note that this symmetry prehandling structure admits some degrees of
freedom in selecting~$\varphi_\beta$.
If~$\varphi_\beta = \id$ for all $\beta \in \bnbnodes$, this resembles the
adapted version of orbitopal fixing as mentioned in
Section~\ref{sec:orbitopalfixing}.
But also other choices are possible as we will discuss in
Section~\ref{sec:num}.

\subsection{Variable ordering derived from branch-and-bound}
\label{sec:gen:orbitalreduction}
A natural question is whether also generalizations for isomorphism pruning and
\orbitalfixing exist.
The main challenge is that after branching on general variables, they are
not necessarily fixed (in contrast to the binary setting).
Thus, stabilizer computations as discussed in Section~\ref{sec:overview} might
not apply in the generalized setting.
Inspired by \orbitalfixing, we present a way to reduce variable domains of
arbitrary variables based on symmetry,
called \emph{orbital reduction}.

For vectors $x$ and $y$ of equal length $m$, we write $x \leq y$
if $x_i \leq y_i$ for all $i \in [m]$.
Let~$\beta \in \bnbnodes$ be a node of the branch-and-bound tree,
$W^\beta \define 
\{ x \in \mathds R^n : 
\sigma_\beta(x) \succeq \sigma_\beta(\delta(x))\ 
\text{for all}\ \delta \in \Gamma 
\}$,
and~$\Delta^\beta \define \{
\gamma \in \Gamma
:
\sigma_\beta(x) \leq \sigma_\beta \gamma(x)\
\text{for all}\ x \in \feasregion(\beta) \cap W^\beta
\}
$
be a group of symmetries.
Similar to Section~\ref{sec:overviewtree},
we intend to use $\Delta^\beta$ to find \VDRs.
We first show that this is indeed a group.

\begin{lemma}
Let~$\beta$ be a node of a \bb-tree using single variable branching
with symmetry prehandling structure of Example~\ref{ex:vardynamic}.
Then, $\Delta^\beta$ is a group.
\end{lemma}
\begin{proof}
  Recall that we assume~$\Gamma \leq \symmetricgroup{n}$  in this
  section.
  Therefore, both~$\Gamma$ and~$\Delta^\beta \subseteq \Gamma$ are finite.
To show that it is a group, it suffices to show that
compositions of elements of $\Delta^\beta$ are also contained
therein. The identity and inverses follow implicitly.

Let $\gamma_1, \gamma_2 \in \Delta^\beta$,
and suppose $x \in \feasregion(\beta) \cap W^\beta$.
By definition of $\Delta^\beta$,
we have $\sigma_\beta(x) \leq \sigma_\beta(\gamma_2(x))$.
Since $\gamma_2 \in \Delta^\beta \leq \Gamma$
and $x \in W^\beta$, we have
$\sigma_\beta(x) \succeq \sigma_\beta(\gamma_2(x))$.
Note that 
$\sigma_\beta(x) \succeq \sigma_\beta(\gamma_2(x))$
and~%
$\sigma_\beta(x) \leq \sigma_\beta(\gamma_2(x))$
imply~%
$\sigma_\beta(x) = \sigma_\beta(\gamma_2(x))$.
Since the correctness conditions are satisfied for Example~\ref{ex:vardynamic},
due to Condition~\eqref{cond:permutationcondition},
the properties $x \in \feasregion(\beta)$ and 
$\sigma_\beta(x) = \sigma_\beta(\gamma_2(x))$
imply that~$\gamma_2(x) \in \feasregion(\beta)$ holds.

Since $x \in W^\beta$,
for all~$\delta \in \Gamma$
we have~%
$\sigma_\beta(\gamma_2(x)) = \sigma_\beta(x) 
\succeq \sigma_\beta(\delta(x))$.
Because $\gamma_2$ is a group element of $\Gamma$,
we thus also have 
$\sigma_\beta(\gamma_2(x)) 
\succeq \sigma_\beta(\delta \circ \gamma_2(x))$
for all $\delta \in \Gamma$,
meaning that $\gamma_2(x) \in W^\beta$.
Summarizing, we have $\sigma_\beta(x) = \sigma_\beta(\gamma_2(x))$
and $\gamma_2(x) \in \feasregion(\beta) \cap W^\beta$.
By analogy, the same results hold for~$\gamma_1(x)$.

Since $\gamma_2(x) \in \feasregion(\beta) \cap W^\beta$
and $\gamma_1 \in \Delta^\beta$,
the definition of $\Delta^\beta$ yields
$\sigma_\beta(\gamma_2(x)) \leq \sigma_\beta(\gamma_1 \circ \gamma_2(x))$.
Using the same reasoning as above, $\gamma_2(x) \in W^\beta$ 
implies~%
$\sigma_\beta(\gamma_2(x)) \succeq \sigma_\beta(\gamma_1 \circ \gamma_2(x))$,
so~$\sigma_\beta(\gamma_2(x)) = \sigma_\beta(\gamma_1 \circ \gamma_2(x))$.
We thus find that 
$\sigma_\beta(x) = \sigma_\beta(\gamma_2(x)) 
= \sigma_\beta(\gamma_1 \circ \gamma_2(x))$,
which implies $\gamma_1 \circ \gamma_2 \in \Delta^\beta$.
\end{proof}

We show two feasible reductions that are based on $\Delta^\beta$.
The first reduction shows that the variable domains of the variables
in the $\Delta^\beta$-orbit of the branching variable can possibly be tightened.
To this end, 
denote the orbit of $i$ in $\Delta^\beta$ by~$O_i^\beta \define \{
\gamma(i) : \gamma \in \Delta^\beta \}$.
If the branching decision after node $\beta$
decreased the upper bound of variable~$x_i$ for some~$i \in [n]$,
a valid \VDR is to decrease the upper bounds of~$x_j$ for all $j \in O^\beta_i$
to the same value as we show next.

\begin{lemma}[Orbital symmetry handling]
\label{lem:branchorbit}
Let $\bnbtree = (\bnbnodes, \bnbedges)$ 
be a \bb-tree using single variable branching
with symmetry prehandling structure of Example~\ref{ex:vardynamic}.
Let $\omega \in \bnbnodes$ be a child of $\beta \in \bnbnodes$
where~$x_i$ is the branching variable for some $i \in [n]$.
Then, at node~$\omega$, each solution~$x \in \feasregion(\omega)$ 
satisfying~${\sigma_\omega(x) \succeq \sigma_\omega(\delta(x))}$
for all $\delta \in \Gamma$
(i.e., \eqref{eq:main} for node $\omega$) also
satisfies~$x_i \geq x_j$ for all~$j \in O_i^\beta$.
\end{lemma}

\begin{proof}
Let $\gamma \in \Delta^\beta$
and let $x \in \feasregion(\omega)$ satisfy
$\sigma_\omega(x) \succeq \sigma_\omega(\delta(x))$ 
for all $\delta \in \Gamma$.
Since $\omega$ is a child of $\beta$, 
we have~$x \in \feasregion(\omega) \subseteq \feasregion(\beta)$.
Also, for all $\delta \in \Gamma$
we have $\sigma_\omega(x) \succeq \sigma_\omega(\delta(x))$,
so due to
the symmetry prehandling structure of Example~\ref{ex:vardynamic},
we also have~$\sigma_\beta(x) \succeq \sigma_\beta(\delta(x))$,
meaning that $x \in W^\beta$.
Hence, we have $x \in \feasregion(\beta) \cap W^\beta$.
By definition of $\Delta^\beta$,
we thus have $\sigma_\beta(x) \leq \sigma_\beta(\gamma(x))$.
Recall that due to Example~\ref{ex:vardynamic},
we have for all $\delta \in \Gamma$ that
$\sigma_\beta(x) \succeq \sigma_\beta(\delta(x))$.
Since $\gamma \in \Delta^\beta \leq \Gamma$,
therefore~$\sigma_\beta(x) \leq \sigma_\beta(\gamma(x))$
and
$\sigma_\beta(x) \succeq \sigma_\beta(\gamma(x))$ hold,
implying 
$\sigma_\beta(x) = \sigma_\beta(\gamma(x))$.
Denote this result by ($\dagger$).

Due to Example~\ref{ex:vardynamic}, we have
\[
\sigma_\omega(x) =
\begin{cases}
\sigma_\beta(x),
& \text{if}\ i \in (\pi_\beta \varphi_\beta)^{-1}(m_\beta)
\ \text{(i.e., variable $x_i$ appears in $\sigma_\beta(x)$), and}\\
\binom{\sigma_\beta(x)}{x_i},
& \text{otherwise.}
\end{cases}
\]
As such, \SHC $\sigma_\omega (x) \succeq \sigma_\omega (\gamma(x))$
is equivalent to either
$\sigma_\beta(x) \succ \sigma_\beta (\gamma(x))$,
or both~$\sigma_\beta(x) = \sigma_\beta (\gamma(x))$
and $x_i \geq \gamma(x)_i$.
Note that this statement is the case independent from whether 
entry $i$  is branched on before or not 
(i.e., whether $i \in (\pi_\beta\varphi_\beta)^{-1}([m_\beta])$ or not).
Using ($\dagger$), the first of the two options cannot hold,
so we must have~$x_i \geq \gamma(x)_i = x_{\gamma^{-1}(i)}$.
Consequently, $x_i \geq x_{\gamma^{-1}(i)}$ is a valid \SHC
for~$\gamma \in \Delta^\beta$.
Thus, for all $j \in O_i^\beta$, we can propagate $x_i \geq x_j$.
\end{proof}

Second, recall our assumption that any \VDR that is not based on our
symmetry framework needs to be symmetry compatible, see
Section~\ref{sec:framework}.
In practice, however, a solver might not find all symmetric \VDRs, e.g.,
due to iteration limits.
The following lemma allows us to find missing (but not necessarily all)
\VDRs based on symmetries, which corresponds to orbital fixing as discussed
in~\cite{pfetsch2019computational} without the restriction to binary
variables.

\begin{lemma}
\label{lem:intersection}
Let~$\beta$ be a node of a \bb-tree using single variable branching
with symmetry prehandling structure of Example~\ref{ex:vardynamic}.
Then, when \SHCs~\eqref{eq:main} are enforced
(i.e., solutions are in $W^\beta$),
a valid \VDR is to reduce the domain of~$x_i$ to the
intersection of all variable domains $x_j$ for $j \in O_i^\beta$.
\end{lemma}
\begin{proof}
Let $x \in \feasregion(\beta)$ and let~$i \in [n]$.
Let~$j \in O_i^\beta$, i.e., there exists~$\gamma \in \Delta^\beta$
with~$\gamma(i) = j$.
As $\gamma \in \Delta^\beta$, 
$\sigma_\beta(x) \leq \sigma_\beta(\gamma(x))$ holds.
Since \SHCs~\eqref{eq:main} are enforced,
it must as well hold that
$\sigma_\beta(x) \succeq \sigma_\beta(\gamma(x))$,
and thus
$\sigma_\beta(x) = \sigma_\beta(\gamma(x))$.
Since Example~\ref{ex:vardynamic} satisfies the correctness conditions,
Condition~\eqref{cond:permutationcondition}
yields~$\gamma(x) \in \feasregion(\beta)$.
Thus, $x_i$ is not only contained in the domain of variable~$i$, but also
in the domain of variable~$x_{\gamma(i)}$.
For this reason, the domain of~$x_i$ can be restricted to the
intersection of the domains for all variables $x_j$ for all $j \in O_i^\beta$.
\end{proof}
In practice, Lemma~\ref{lem:branchorbit} and~\ref{lem:intersection} cannot
be used immediately as the orbits depend on~$\Delta^\beta$, which cannot be
computed easily as it depends on~$\feasregion(\beta)$ and $W^\beta$.
Instead, we base ourselves on a 
suitably determined subgroup~$\tilde \Delta^\beta \leq \Delta^\beta$,
and apply the reductions induced by that subgroup.
Because the reductions are based on variables in the same orbit,
and orbits of the subgroup are subsets of the orbits of the larger group,
\VDRs found by $\tilde \Delta^\beta$ would also be found by~$\Delta^\beta$.

For all $i \in [n]$, 
let~$\domain_i^\beta \subseteq \mathds R$
be the (known) domain of variable~$x_i$ at node~$\beta$.
In particular, we thus have~$
\feasregion(\beta) \cap W^\beta \subseteq
\feasregion(\beta) \subseteq 
\bigtimes_{i \in [n]} \domain_i^\beta
$.
By replacing $\feasregion(\beta) \cap W^\beta$ 
in the definition of $\Delta^\beta$
by $\bigtimes_{i \in [n]} \domain_i^\beta$,
we get a subset of symmetries: 
$
\{ \gamma \in \Gamma : 
\sigma_\beta(x) \leq \sigma_\beta(\gamma(x))\ 
\text{for all}\ 
x \in \bigtimes_{i \in[n]} \domain_i^\beta
\}
\subseteq \Delta^\beta
$.
In particular, using (a subset of) the left set
as generating set, one finds a permutation group 
that is a subgroup of~$\Delta^\beta$.
For the computational results 
shown in Section~\ref{sec:num} we discuss the subset selection
procedure that we chose in our implementation.

We finish this section by showing that, in binary problems,
\VDRs yielded by \orbitalfixing from 
Section~\ref{sec:overviewtree} are also yielded
by the generalized setting of 
Lemma~\ref{lem:branchorbit} and~\ref{lem:intersection}.

\begin{lemma}
Denote $\Delta^{\smash\beta}_{\smash{\rm bin}}$
as the group $\Delta^{\smash\beta}$  
defined in Section~\ref{sec:overviewtree},
and $\Delta^{\smash\beta}_{\smash{\rm gen}}$
as the group with the same symbol defined here.
If the symmetry group acts on binary variables exclusively,
then $\Delta^\beta_{\rm bin} \leq \Delta^\beta_{\rm gen}$.
\end{lemma}
\begin{proof}
In binary problems, branching on variables fixes their values.
As such, because vector $\sigma_\beta(x)$ contains all branched variables,
it is the same for all~${x \in \feasregion(\beta)}$.
Suppose $\gamma \in \Delta^\beta_{\rm bin}$,
i.e., $\gamma \in \Gamma$ and~$\gamma(B_1^\beta) = B_1^\beta$.
For all $i \in [m_\beta]$, with
$\sigma_\beta(x)_i = 1$,
we have $\sigma_\beta(\gamma(x))_i = 1$
for all $x \in \feasregion(\beta)$.
Similarly, for all $i \in [m_\beta]$
with $\sigma_\beta(x)_i = 0$,
we have $\sigma_\beta(\gamma(x))_i \geq 0$.
This means that for all~$x \in \feasregion(\beta)$
we have $\sigma_\beta(x) \leq \sigma_\beta(\gamma(x))$.
This holds in particular for all $x \in \feasregion(\beta) \cap W^\beta$,
i.e., $\gamma \in \Delta^\beta_{\rm gen}$.
\end{proof}

By the \orbitalfixing rule described in Section~\ref{sec:overviewtree},
for all variable indices $i$ where $x_i$ is branched to zero,
all variables $x_j$ with $j$ in the orbit of $i$ in 
$\Delta^\beta_{\smash{\rm bin}}$ can be fixed to zero, as well.
Because $\Delta^\beta_{\smash{\rm bin}} 
\leq \Delta^\beta_{\smash{\rm gen}}$,
every orbit of $\Delta^\beta_{\smash{\rm bin}}$
is contained in an orbit of $\Delta^\beta_{\smash{\rm gen}}$.
As such, if $x_i$ is the branching variable at the present node, 
this is implied by~$x_i \geq x_j$ in Lemma~\ref{lem:branchorbit}.
Otherwise, this is implied by Lemma~\ref{lem:intersection}.

\section{Computational study}
\label{sec:num}

To assess the effectiveness of our methods, we compare the running times
of the implementations of the various dynamic symmetry handling methods
of Section~\ref{sec:cast} (in the regime of Examples~\ref{ex:vardynamic}
and~\ref{ex:orbitopalfixing}) to similar existing methods.
To this end, we make use of diverse testsets.
\begin{itemize}
\item Symmetric benchmark instances from
MIPLIB~2010~\cite{KochEtAl2011MIPLIB} and
MIPLIB~2017~\cite{Gleixner2021MIPLIB}.
\item Existence of minimum $t\text{-}(v,k,\lambda)$-covering designs.
\item Noise dosage problem instances as discussed by 
Sherali and Smith~\cite{sherali2001models} (cf. Problem~\ref{prob:NDbinary}).
\end{itemize}
The MIPLIB instances offer a diverse set of instances that contain
symmetries, but these symmetries operate on binary variables predominantly.
To evaluate the effectiveness of our framework for non-binary problems, we
consider the covering design and noise dosage instances.
The symmetries of the former are orbitopal, whereas the latter has no
orbitopal symmetries.

Although our framework allows to handle more general symmetries, we
restrict the numerical experiments to permutation symmetries.
On the one hand, \code{SCIP} can only detect permutation symmetries at this
point of time.
On the other hand, most symmetry handling methods discussed in the literature only
apply to permutation symmetries.
The development of methods for other kinds of symmetries is out of scope of
this article.

\subsection{Solver components and configurations}
We use a development version of 
the solver \code{SCIP~8.0.3.5}~\cite{BestuzhevaEtal2021OO},
commit 
\texttt{8443db2}\footnote{Public mirror: \url{https://github.com/scipopt/scip/tree/8443db213892153ff2e9d6e70c343024fb26968c}},
with LP-solver \code{Soplex~6.0.3}.
\code{SCIP} contains implementations of the state-of-the-art methods
\lexfix, orbitopal fixing, and \orbitalfixing to which we compare our methods.
We have extended the code with our dynamic methods.
Our modified code is available on 
GitHub\footnote{Project page: \url{https://github.com/JasperNL/scip-unified}}.
This repository also contains the instance generators and 
problem instances for the noise dosage and covering design problems.

For all settings, symmetries are detected
by finding automorphisms of a suitable
graph~\cite{pfetsch2019computational,salvagnin2005dominance} using
\code{bliss~0.77}~\cite{JunttilaKaski2015bliss}.
We make use of the readily implemented symmetry detection code of \code{SCIP},
which finds a set of permutations $\Pi$ that generate a
symmetry group $\Gamma$ of the problem, namely the symmetries implied
by its formulation~\cite{pfetsch2019computational,salvagnin2005dominance,margot2010symmetry}.
This is a permutation group acting on the solution vector index space,
so the setting of Section~\ref{sec:overview} and~\ref{sec:cast} applies.

If~$\Gamma$ is a product group consisting of $k$ components, i.e., $\Gamma
= \bigtimes_{i \in [k]} \Gamma_i$,
then by using similar arguments 
as in~\cite[Proposition~5]{hojny2019polytopes},
the symmetries of the different components~$\Gamma_i$, $i \in [k]$, can be
handled independently; compositions of permutations from different
components do not need to be taken into account.
In particular, it is possible to select a different symmetry prehandling
structure for the different components.
We therefore decompose the set of permutations $\Pi$ generating $\Gamma$
that are found by \code{SCIP} in components,
yielding generating sets~$\Pi_1, \dots, \Pi_k$ 
for components~$\Gamma_1, \dots, \Gamma_k$.
Symmetry in each component is handled separately.

For all settings, we disable restarts to ensure that all methods exploit
the same symmetry information.
We compare our newly implemented methods to the
methods originally implemented in \code{SCIP}.

\paragraph{Our configurations}
For every component $\Gamma_i$, 
we handle symmetries as follows, where we skip some steps if the
corresponding symmetry handling method is disabled.
\begin{enumerate}
\item\label{step:orbitope} If a \code{SCIP}-internal heuristic,
  cf.~\cite[Sec.~4.1]{hojny2019polytopes}, detects that~$\Gamma_i$ consists
  of orbitopal symmetries:
  \begin{enumerate}
  \item\label{step:cons} If the orbitope matrix is a single row 
    $[x_1 \cdots x_\ell]$, add linear constraints
    $x_1 \geq \dots \geq x_\ell$.
  \item Otherwise, if the orbitope matrix contains only two columns 
    (i.e., is generated by a single permutation $\gamma$),
    then use lexicographic reduction using the dynamic
    variable ordering of Example~\ref{ex:vardynamic},
    as described in Section~\ref{sec:gen:lexred}.
  \item\label{step:pack} Otherwise, if there are at least 3 rows with 
    binary variables whose sum is at most 1 (so-called packing-partitioning type)
    then use the complete static propagation method for packing-partitioning
    orbitopes as described by Kaibel and Pfetsch~\cite{KaibelPfetsch2008},
    where the orbitope matrix is restricted to the rows with this structure.
  \item\label{step:dyn} Otherwise, use dynamic orbitopal reduction
    as described in Section~\ref{sec:gen:orbitopalfixing} 
    using the dynamic variable ordering of Example~\ref{ex:orbitopalfixing}.
    We select $\varphi_\beta$ such that it
    swaps the column containing the branched variable
    to the middlemost (or leftmost) symmetrically equivalent column
    when propagating the \SHCs~\eqref{eq:main}
    using static orbitopal reduction.
  \end{enumerate}
\item Otherwise (i.e., if the symmetries are not orbitopal),
  use the symmetry prehandling structure of Example~\ref{ex:vardynamic}
  and use two compatible methods simultaneously:
  \begin{enumerate}
  \item Lexicographic reduction as described 
    in Section~\ref{sec:gen:lexred}.
  \item Orbital reduction as described 
    in Section~\ref{sec:gen:orbitalreduction}.
    Since computing $\Delta^\beta$ is non-trivial,
    we work with a subgroup of $\Delta^\beta$,
    namely the group generated by all permutations $\gamma \in \Pi_i$
    for which $\sigma_\beta(x) \leq \sigma_\beta(\gamma(x))$
    for all $x \in \bigtimes_{j \in [n]} \domain_j^\beta$.
  \end{enumerate}
\end{enumerate}

We also compare settings where 
orbitopal reduction, orbital reduction and lexicographic reduction
are turned off. If orbitopal symmetries are not handled, we always
resort to the second setting, where lexicographic reduction (if enabled)
and/or orbital reduction (if enabled) are applied.

For orbitopal symmetries, we have chosen to handle certain common cases
before refraining to the dynamic orbitopal reduction code that we devised.
First, for single-row orbitopes, the symmetry is completely handled
by the linear constraints.
Since linear constraints are strong and work well with other components
of the solver, we decided to handle those this way.
If an orbitope only has two columns, the underlying symmetry group
is generated by a single permutation of order~2. In that case,
the symmetry is completely handled by lexicographic reduction.
Third, it is well known that exploiting problem-specific information
can greatly assist symmetry handling.
If a packing-partitioning structure can be detected, we therefore apply
specialized methods as discussed above.
Otherwise, we use orbitopal reduction as discussed in
Section~\ref{sec:gen:orbitopalfixing}.
In Step~\ref{step:dyn}, the choice
of~$\varphi_\beta$ is inspired by the discussion in
Section~\ref{sec:theframework} (``Interpretation'').
By moving the branching variable to the middlemost possible column,
balanced subproblems are created, whereas the leftmost possible column
might lead to more reductions in one child than in the other.
Below, we will investigate which technique is more favorable.

For the components that are do not consist of orbitopal symmetries,
we decided to settle on the setting of Example~\ref{ex:vardynamic}
and use both compatible methods.
Since the \code{SCIP} version that we compare to either uses orbital fixing or
(static) lexicographic fixing for such components, our setting allows us to
assess the impact of adapting lexicographic fixing to make it compatible
with orbital fixing.

\paragraph{Comparison base}
We compare to similar, readily implemented methods in \code{SCIP}.
In \code{SCIP} jargon, these methods are called \emph{polyhedral} and
\emph{orbital fixing} and can be enabled/disabled independently.
The polyhedral methods consist of \lexfix for the \SHCs~$x \succeq
\gamma(x)$ for~$\gamma \in \Pi_i$ and methods to handle orbitopal
symmetries; orbital fixing uses \orbitalfixing
from~\cite{pfetsch2019computational}.
Note that only symmetries of binary variables are handled.

If a component consists of polytopal symmetries, the polyhedral methods
handle these symmetries by carrying out the check of Step~\ref{step:pack}.
In case it evaluates positively, methods exploiting packing-partitioning
structures are applied as described above.
Otherwise, orbitopal symmetries of binary variables are handled by a
variant of row-dynamic orbitopal fixing.
The remaining components are either handled by static \lexfix or orbital
fixing, depending on whether the polyhedral methods or orbital fixing is
enabled.
Moreover, if polyhedral methods are disabled, orbital fixing is also
applied to components consisting of orbitopal symmetries.

\subsection{Results}
All experiments have been run in parallel
on the Dutch National supercomputer 
Snellius ``thin'' consisting of compute nodes with dual AMD Rome 7H12
processors providing a total of \num{128} physical CPU cores,
and~\SI{256}{\giga\byte} memory.
Each process has an allocation of \num{4} physical CPU cores
and~\SI{8}{\giga\byte} of memory.
In the results below we report the running times 
(column \emph{time})
and time spent on symmetry handling (column \emph{sym}) in shifted 
geometric mean~$\prod_{i = 1}^n (t_i + 1)^{\frac{1}{n}} - 1$ 
to reduce the impact of outliers.
We also report the number of instances solved within the time limit 
of~\SI{1}{\hour} (column~\emph{\#S}). 
If the time limit is reached, the solving time of that instance
is reported as \SI{1}{\hour}.
None of the instances failed or exceeded the memory limit.
We report the aggregated results for all instances,
for all instances for which at least one setting solved the instance
within the time limit,
and for all instances solved by all settings.
For each of these classes, we provide their size below.

We use abbreviations for the settings. We compare
no symmetry handling (\emph{Nosym}),
traditional polyhedral methods (\emph{Polyh}),
traditional orbital fixing (\emph{OF}),
dynamic orbitopal reduction (\emph{OtopRed}),
dynamic lexicographic reduction (\emph{LexRed}),
orbitopal reduction (\emph{OR}),
and combinations hereof.
Note that also setting \emph{Nosym} reports a small symmetry time,
because due to \code{SCIP}'s architecture, handling the corresponding
plug-in requires some time even if it is not used.
Moreover, if symmetries are handled by linear constraints in the model
(cf.\ Step~\ref{step:cons}),
these are not reported in the symmetry handling figures.
Recall from~\ref{step:dyn} that we consider two variants of
selecting~$\varphi_\beta$ for dynamic orbitopal fixing.
We refer to these variants as \emph{first} and \emph{median} 
for the leftmost and
middlemost column, respectively.
As the noise dosage testset exclusively consists of orbitopal symmetries,
we test both parameterizations there.
For the remaining testsets, we restrict ourselves to \emph{median} as it
performs better on average for the noise dosage instances.

Since the covering design and noise dosage testsets
are relatively small and contain many easy instances, performance
variability is minimized by repeating each 
configuration-instance pair three times with a different
global random seed shift.
Due to the large number of instances, settings and large time requirements,
only one seed is used for the MIPLIB instances.

\subsubsection{MIPLIB}
\label{sec:results:miplib}
To compose our testset, we presolved all instances from the
MIPLIB~2010~\cite{KochEtAl2011MIPLIB} and
MIPLIB~2017~\cite{Gleixner2021MIPLIB} benchmark testsets
and selected those for which \code{SCIP} could detect symmetries.
This results in~129 instances.
We excluded instance \texttt{mspp16} as it exceeds the memory limit
during presolving.

The goal of our experiments for MIPLIB instances is twofold.
On the one hand, we investigate whether our framework allows to solve
symmetric mixed-integer programs faster than the state-of-the-art methods
as implemented in \code{SCIP}.
On the other hand, we are interested in the effect of adapting different
symmetry handling methods for~$\sigma_\beta(x) \succeq \sigma_\beta
\gamma(x)$ in comparison with their static counterparts.
Table~\ref{tab:results:miplib} shows the aggregate results of our
experiments.

\begin{table}[t]
\caption{Results for MIPLIB 2010 and MIPLIB 2017}
\label{tab:results:miplib}
\centering
\footnotesize
\begin{tabular}{@{}L{3.2cm}*{3}{R{1cm}R{1cm}R{.6cm}}@{}}
\toprule
\multicolumn{1}{c}{Setting}
& \multicolumn{3}{c}{All instances (128)}
& \multicolumn{3}{c}{Solved by some setting (75)}
& \multicolumn{3}{c}{Solved by all settings (51)}
\\
\cmidrule(lr){1-1}
\cmidrule(lr){2-4}
\cmidrule(lr){5-7}
\cmidrule(lr){8-10}
& time (s) & sym (s) & \#S
& time (s) & sym (s) & \#S
& time (s) & sym (s) & \#S
\\
\cmidrule{2-10}
Nosym
&   970.73 &     0.18 &   59
&   384.01 &     0.17 &   59
&   157.61 &     0.08 &   51
\\
Polyh
&   747.43 &     1.82 &   67
&   245.46 &     1.02 &   67
&   138.09 &     0.69 &   51
\\
OF
&   822.26 &     1.34 &   66
&   289.16 &     0.93 &   66
&   134.34 &     0.48 &   51
\\
Polyh + OF
&   728.88 &     1.97 &   69
&   235.17 &     0.98 &   69
&   130.08 &     0.58 &   51
\\
\midrule
OtopRed
&   811.90 &     1.34 &   62
&   282.95 &     0.75 &   62
&   129.08 &     0.49 &   51
\\
LexRed
&   799.82 &     2.22 &   67
&   275.78 &     1.30 &   67
&   142.30 &     0.73 &   51
\\
OR
&   807.75 &     4.94 &   66
&   280.46 &     3.29 &   66
&   140.62 &     1.40 &   51
\\
OR + LexRed
&   788.29 &     5.44 &   68
&   269.01 &     3.43 &   68
&   138.18 &     1.57 &   51
\\
OR + OtopRed
&   727.04 &     2.38 &   66
&   234.14 &     1.34 &   66
&   128.62 &     0.57 &   51
\\
OtopRed + LexRed
&   691.22 &     1.90 &   68
&   214.84 &     1.01 &   68
&   123.68 &     0.56 &   51
\\
OR + OtopRed + LexRed
&   708.63 &     2.57 &   67
&   224.18 &     1.43 &   67
&   125.73 &     0.63 &   51
\\
\bottomrule
\end{tabular}%
\end{table}

Regarding the first question, we observe that using any type of symmetry
handling is vastly superior over the setting where no symmetry is handled.
Considering all instances,
the best of the traditional settings reports an average running
time improvement of \SI{24.9}{\percent} over \emph{Nosym},
and the best of the dynamified methods report \SI{28.8}{\percent}.
Our framework thus allows to improve on \code{SCIP}'s state-of-the-art
by~\SI{5.2}{\percent}.
On the instances that can be solved by at least one setting, this effect is
even more pronounced and improves on \code{SCIP}'s best setting by
\SI{8.6}{\percent}.
We believe that this is a substantial improvement, because the MIPLIB
instances are rather diverse.
In particular, some of these instances contain only very few
symmetries.

Regarding the second question, we compare the \code{SCIP} settings
\emph{Polyh}, \emph{OF}, and \emph{Polyh + OF} with their counterparts in our
framework, being \emph{OtopRed + LexRed}, \emph{OR}, and \emph{OR +
  OtopRed}, respectively.
The running time of the pure polyhedral setting \emph{Polyh} can be
improved in our framework by~\SI{7.5}{\percent} when considering all
instances, and by~\SI{12.4}{\percent} when considering only the instances
solved by some setting within the time limit.
Consequently, adapting symmetry handling to the branching order via the
symmetry prehandling structure of Example~\ref{ex:vardynamic}
and~\ref{ex:orbitopalfixing} allows to gain substantial performance
improvements.
Our explanation for this behavior is that symmetry reductions can be found
much earlier in branch-and-bound than in the static setting (cf.\
Figure~\ref{fig:branching:orbitopalfixing}, where no reductions can be
found at depth~1 if no adaptation is used).
Thus, symmetric parts of the branch-and-bound tree can be pruned earlier.

Comparing \emph{OF} and \emph{OR}, we observe that \emph{OR} is slightly
faster than \emph{OF}.
Both methods, however, are much slower than \emph{Polyh} and \emph{OtopRed
  + LexRed}.
A possible explanation is that the latter methods make use of orbitopal
fixing, which can handle entire symmetry groups, whereas \emph{OF} and
\emph{OR} only find some reductions based on orbits.

Among \code{SCIP}'s methods, \emph{Polyh + OF} performs best.
Its counterpart \emph{OR + OtopRed} in our framework performs comparably on
all instances and the solvable instances, however, three fewer instances
are solved.
A possible explanation for the comparable running time is that traditional
orbital fixing and the variant of row-dynamic
orbitopal fixing are already dynamic methods.
Thus, a comparable running time can be expected (although this does not
explain the difference in the number of solved instance).

Lastly, we discuss settings which are not possible in the traditional
setting, i.e., combing \lexred and orbital reduction.
Enhancing orbital reduction by \lexred indeed leads to an improvement
of~\SI{2.4}{\percent} and allows to solve two more instances.
The best setting in our framework, however, does not enable all methods in
our framework.
Indeed, the running time of \emph{OR + OtopRed + LexRed} can be improved by
\SI{4.9}{\percent} when disabling orbital reduction.
We explain this phenomenon with the fact that orbitopal reduction already
handles a lot of group structure.
Combining \lexred and orbital reduction on the remaining components only
finds a few more reductions.
The time needed for finding these reductions is then not compensated by the
symmetry reduction effect.
If no group structure is handled via orbitopal reduction, \lexred can
indeed be enhanced by orbital reduction.

Recall that symmetries in MIPLIB instances
act predominantly on binary variables.
As opposed to the traditional settings that we compare to,
the generalized setting can handle symmetries on non-binary
variable domains.
Thus, potentially more reductions can follow from larger 
symmetry groups that include non-binary variables.
As shown in Appendix~\ref{app:tables},
a larger group is detected in only 10 out of the 128
instances, and only 2 of these can be solved by some setting.
When considering the subset of instances
where symmetry is handled based on the same group,
we report similar results as before.
This shows that even if we only consider problems
where symmetries act on the binary variables,
our generalized methods outperform similar state-of-the-art methods.

\subsubsection{Minimum $\boldsymbol{t\text{-}(v, k, \lambda)}$-covering
  designs}

Since the symmetries of MIPLIB instances predominantly act on binary
variables, we turn the focus in the following to symmetric problems without
binary variables to assess our framework in this regime.
Let~$v \geq k \geq t \geq 0$ and $\lambda > 0$ be integers. 
Let $V$ be a set of cardinality $v$, 
and let~$\mathcal K$ (resp.\@~$\mathcal T$)
be the collections of all subsets of~$V$ having sizes~$k$ (resp.\@~$t$).
A~\emph{$t\text{-}(v, k, \lambda)$-covering design} is a
multiset~$\mathcal{C} \subseteq \mathcal{K}$
if all sets $T \in \mathcal T$ are contained in 
at least $\lambda$ sets of $\mathcal C$, counting with multiplicity.
A covering design is \emph{minimum} if no smaller covering design
with these parameters exist, and finding these is of interest,
e.g., in~\cite{margot2003coveringdesigns,nurmela1999covering,fadlaoui2011tabucoveringdesigns}.
Margot~\cite{margot2003coveringdesigns} gives an ILP-formulation,
having decision variables~%
$\nu \in \{0, \dots, \lambda \}^{\mathcal K}$
specifying the multiplicity of the sets $K \in \mathcal K$
for the minimum $t\text{-}(v, k, \lambda)$-covering design sought after.
The problem is to
\begin{subequations}
\makeatletter
\def\@currentlabel{CD}
\makeatother
\renewcommand{\theequation}{CD\arabic{equation}}%
\label{eq:CD}
\begin{align}
\text{minimize}\ \sum_{K \in \mathcal K} \nu_{K}&, \\
\text{subject to}\ \sum_{K \in \mathcal K : T \subseteq K} \nu_K &\geq \lambda
&&\text{for all}\ T \in \mathcal T,\\
\nu &\in \{ 0, \dots, \lambda \}^{\mathcal K}.
\end{align}
\end{subequations}
Symmetries in this problem re-label the elements of $V$,
and these are also detected by the symmetry detection routine of \code{SCIP}.
Note that, although the underlying group is symmetric,
these symmetries in terms of the variables $\nu$ are not orbitopal.

Margot~\cite{margot2003coveringdesigns} considers
an instance with $\lambda = 1$, which is a binary problem.
We consider all non-binary instances with parameters 
$\lambda \in \{ 2, 3 \}$ and $12 \geq v \geq k \geq t \geq 1$,
and restricted ourselves to instances that were solved
within \SI{7200}{\second} in preliminary runs
and that require at least \SI{10}{\second} for some setting.
This way, 58 instances remain.
With 3 seeds per instance, we end up with~174 results.
The aggregated results are shown in Table~\ref{tab:results:coveringdesigns}.
Because the instances are non-binary,
none of the considered traditional methods can be applied to handle
the symmetries. As such, we can only compare to no symmetry handling.
The best of our instances reports an improvement of~\SI{77.8}{\percent}
over no symmetry handling.
\begin{table}[t]
\caption{Results for finding minimum $t\text{-}(v,k,\lambda)$ covering designs}
\label{tab:results:coveringdesigns}
\centering
\footnotesize
\begin{tabular}{@{}L{2.4cm}*{3}{R{1cm}R{1cm}R{.6cm}}@{}}
\toprule
\multicolumn{1}{c}{Setting}
& \multicolumn{3}{c}{All instances (174)}
& \multicolumn{3}{c}{Solved by some setting (171)}
& \multicolumn{3}{c}{Solved by all settings (126)}
\\
\cmidrule(lr){1-1}
\cmidrule(lr){2-4}
\cmidrule(lr){5-7}
\cmidrule(lr){8-10}
& time (s) & sym (s) & \#S
& time (s) & sym (s) & \#S
& time (s) & sym (s) & \#S
\\
\cmidrule{2-10}
Nosym
&   211.13 &     0.16 &  138
&   200.85 &     0.15 &  138
&   100.65 &     0.08 &  126
\\
LexRed
&    80.98 &     2.19 &  154
&    75.72 &     2.01 &  154
&    30.67 &     0.60 &  126
\\
OR
&    42.75 &     0.72 &  159
&    39.49 &     0.64 &  159
&    18.52 &     0.20 &  126
\\
OR + LexRed
&    39.92 &     1.51 &  159
&    36.83 &     1.35 &  159
&    17.00 &     0.50 &  126
\\
\bottomrule
\end{tabular}%
\end{table}

If orbital reduction and \lexred are not combined, orbital reduction is the
more competitive method as it improves upon \lexred by~\SI{47.2}{\percent}.
Although \lexred is much faster than \emph{Nosym}, this comparison shows
that more symmetries can be handled when the group structure is exploited
via orbits.
Nevertheless, our framework allows to further improve orbital reduction by
another~\SI{6.6}{\percent} if it is combined with \lexred.
That is, orbital reduction is not able to capture the entire group
structure and missing information can be handled by \lexred efficiently via
our framework.

\subsubsection{Noise dosage}
\label{sec:results:noisedosage}
To assess the effectiveness of orbitopal reduction isolated from other
symmetry handling methods, we consider the noise dosage (ND)
problem~\cite{sherali2001models}, which has orbitopal symmetries as
explained in Problem~\ref{prob:NDbinary}.
However, we replace the binary constraint~\eqref{prob:ND:binary}
by~${\vartheta \in \Z_{\geq 0}^{p \times q}}$ to be able to evaluate the
effect of orbitopal reduction on non-binary problem.
In particular, we are interested whether the choice of the
\emph{left} or \emph{median} variants
matters for the parameter~$\varphi_\beta$.

For each of the parameters $(p, q) \in \{ (3, 8), (4, 9), (5, 10) \}$,
Sherali and Smith have generated four instances~\cite{sherali2001models}.
We thank J.\@ Cole Smith for providing us these instances.
It turns out that these instances are dated and very easy to solve 
even without symmetry handling methods.
As such, we have extended the testset.
For each parameterization
$(p, q) \in \{ (6, 11), (7, 12), \dots, (11, 16) \}$,
we generate five instances.
The details of our generator are in Appendix~\ref{app:noisedosage}.

Symmetries in the ND problem can be handled by adding
$\sum_{i=1}^p M^{p-i} \vartheta_{i,j} \geq \sum_{i=1}^p M^{p-i} \vartheta_{i,j+1}$
for $j \in \{1, \dots, q-1\}$,
where $M$ is an upper bound to the maximal number of tasks
that one worker can perform on a machine,
as described by Sherali and Smith~\cite{sherali2001models}.
This is similar to fundamental domain inequalities~\cite{friedman2007fundamental}
and the symmetry handling constraints of Ostrowski~\cite{ostrowski2009symmetry}.
Although these can be used to handle symmetries,
it is folklore that problems with largely deviating coefficients can lead
to numerical instabilities.
These constraints work well for instances with a small number of machines $p$,
such as in the original instances of Sherali and Smith.
However, for our instance \texttt{noise11\_16\_480\_s2}, such a constraint
has minimal absolute coefficient~1 and maximal absolute coefficient~$11^{10}$.
Various warnings in the log files confirm the presence of numerical
instabilities.

In fact, observable incorrect results follow.
For instance, when adding these linear constraints,
instance \texttt{noise11\_16\_480\_s1} finds infeasibility
during presolving, whereas it is a feasible instance
as illustrated by the no symmetry handling and
dynamic orbitopal fixing runs.
Moreover, for instance \texttt{noise9\_14\_480\_s3},
no infeasibility is detected, but reports a wrong optimal solution.
Thus, there is a need to replace these inequalities with numerically more
stable methods.

\begin{table}[t]
\caption{Results for finding optimal solutions to the noise dosage problems}
\label{tab:results:noisedosage}
\centering
\footnotesize
\begin{tabular}{@{}L{2.4cm}*{3}{R{1cm}R{1cm}R{.6cm}}@{}}
\toprule
\multicolumn{1}{c}{Setting}
& \multicolumn{3}{c}{All instances (165)}
& \multicolumn{3}{c}{Solved by some setting (132)}
& \multicolumn{3}{c}{Solved by all settings (90)}
\\
\cmidrule(lr){1-1}
\cmidrule(lr){2-4}
\cmidrule(lr){5-7}
\cmidrule(lr){8-10}
& time (s) & sym (s) & \#S
& time (s) & sym (s) & \#S
& time (s) & sym (s) & \#S
\\
\cmidrule{2-10}
Nosym
&   152.97 &     1.33 &   90
&    69.02 &     0.87 &   90
&    10.13 &     0.19 &   90
\\
Sherali-Smith
&    26.45 &     0.66 &  129
&     7.11 &     0.18 &  129
&     1.30 &     0.02 &   90
\\
OtopRed (first)
&    35.55 &     4.11 &  123
&    10.60 &     1.29 &  123
&     2.10 &     0.20 &   90
\\
OtopRed (median)
&    33.89 &     3.84 &  117
&     9.95 &     1.25 &  117
&     1.73 &     0.13 &   90
\\
\bottomrule
\end{tabular}%
\end{table}%

We have removed these two instances from our testset as they obviously
result in wrong results.
The aggregated results for the remaining instances are presented in
Table~\ref{tab:results:noisedosage}.
We observe that the symmetry handling inequalities perform~\SI{22.0}{\percent} better than
orbitopal reduction.
However, the presented numbers for the inequality-based approach need
to be interpreted carefully as reporting a correct objective value does not
necessarily mean that the branch-and-bound algorithm worked correctly.
For instance, nodes might have been pruned because of numerical
inaccuracies although a numerical correct algorithm would not have pruned
them.
These issues do not occur for the propagation algorithm orbitopal reduction
as it just reduces variable domains.

Comparing the two parameterizations of orbitopal fixing, we see that the
\emph{median} variant performs \SI{4.7}{\percent} better than the
\emph{first} variant.
This shows that the choice for the column reordering by 
symmetry~$\varphi_\beta$ has a measurable impact on the running time for
dynamic orbitopal reduction.
A possible explanation for the median rule to perform better than the first
rule is that median creates more balanced branch-and-bound trees, cf.\
Section~\ref{sec:theframework}.
Consequently, the right choice of~$\varphi_\beta$ might significantly
change the performance of orbitopal reduction.
We leave it as future research to find a good rule for the selection
of~$\varphi_\beta$ as this rule might be based on the structure of the
underlying problem, e.g., size of the variable domains, number of rows, and
number of columns.

\section{Conclusions and future research}

Symmetry handling is an important component of modern solver technology.
One of the main issues, however, is the selection and combination of
different symmetry handling methods.
Since the latter is non-trivial, we have proposed a flexible framework
that easily allows to check whether different methods are compatible and
that does apply for arbitrary variable domains.
Numerical results show that our framework is substantially faster than
symmetry handling in the state-of-the-art solver \code{SCIP}.
In particular, we benefit from combining different symmetry handling
methods, which is possible in our framework, but only in a limited way in
\code{SCIP}.
Moreover, due to our generalization of symmetry handling algorithms for
binary problems to general variable domains, our framework allows us to
reliably handle symmetries in different non-binary applications.

Due to the flexibility of our framework, it is not only applicable
for the methods discussed in this article, but also allows to apply methods
that will be developed in the future (provided they are compatible with
\SHCs~\eqref{eq:main}).
This opens, among others, the following directions for future research.

In this article, we experimentally evaluated our framework only for
permutation symmetries.
As the framework also supports other types of symmetries
such as rotational and reflection symmetries,
further research could involve devising symmetry handling methods
for such symmetries.
Moreover, to handle permutation symmetries, we only used propagation
techniques.
In the future, these methods can be complemented in two ways.
On the one hand, other techniques such as separation routines can be
applied to handle \SHCs~\eqref{eq:main}.
On the other hand, our symmetry handling methods have not exploited
additional problem structure such as packing-partitioning structures in
orbitopal fixing.
Further research can focus on the incorporation of problem constraints
in handling the symmetry handling constraint of
Theorem~\ref{thm:main}.
This includes a dynamification of packing-partitioning orbitopes,
as well as introducing a way to handle overlapping orbitopal 
subgroups within a component.

Last, in the computational results we describe decision rules for
enabling/disabling certain symmetry handling methods.
If new symmetry handling methods are cast into our
framework, however, these rules need to be updated.
Future research could thus encompass the derivation of good rules for how
to handle symmetries in our framework.

\paragraph{Acknowledgment}
We thank J.\ Cole Smith for providing us the instances of the noise dosage
problem used in~\cite{sherali2001models}.

\clearpage
\appendix
\section{Further results for MIPLIB}
\label{app:tables}

In this appendix, we investigate in more detail the performance gains achieved by our methods for MIPLIB instances in comparison to the traditional methods.
The traditional methods only work for symmetries on binary variables,
whereas our generalized methods can also handle non-binary variables.
As the symmetric MIPLIB benchmark instances contain mostly instances with binary symmetries, the question arises whether the observed gains are thus due to a few instances for which our methods can handle more (non-binary) symmetries.

To investigate this effect, 
we have partitioned these instances in two classes.
The first class contains instances where more symmetries than
in the traditional setting arise;
the second class contains the remaining instances, i.e., both settings
handle the same symmetries.
The results for the subsets of the instances are shown in
Table~\ref{tab:miplib:diff} and~\ref{tab:miplib:same},
respectively.

\begin{table}[b!]
\caption{MIPLIB results where the detected symmetry group is different for the traditional setting and our generalized setting.}
\label{tab:miplib:diff}
\centering
\footnotesize
\begin{tabular}{@{}L{3.2cm}*{3}{R{1.1cm}R{1cm}R{.6cm}}@{}}
\toprule
\multicolumn{1}{c}{Setting}
& \multicolumn{3}{c}{All instances (10)}
& \multicolumn{3}{c}{Solved by some setting (2)}
& \multicolumn{3}{c}{Solved by all settings (1)}
\\
\cmidrule(lr){1-1}
\cmidrule(lr){2-4}
\cmidrule(lr){5-7}
\cmidrule(lr){8-10}
& time (s) & sym (s) & \#S
& time (s) & sym (s) & \#S
& time (s) & sym (s) & \#S
\\
\cmidrule{2-10}
Nosym
&  2397.54 &     0.15 &    1
&   468.79 &     0.02 &    1
&    60.29 &     0.01 &    1
\\
Polyh
&  2369.31 &     4.29 &    1
&   436.46 &     0.01 &    1
&    52.14 &     0.00 &    1
\\
OF
&  2396.04 &     3.09 &    1
&   469.20 &     0.00 &    1
&    60.39 &     0.00 &    1
\\
Polyh + OF
&  2341.74 &     5.69 &    1
&   413.65 &     0.01 &    1
&    46.72 &     0.00 &    1
\\
\midrule
OtopRed
&  2395.33 &     2.63 &    1
&   468.48 &     0.06 &    1
&    60.20 &     0.02 &    1
\\
LexRed
&  2321.62 &     7.21 &    2
&   400.65 &     0.30 &    2
&    46.62 &     0.00 &    1
\\
OR
&  2390.39 &    15.94 &    1
&   463.55 &    12.32 &    1
&    58.93 &     0.17 &    1
\\
OR + LexRed
&  2383.85 &    17.78 &    2
&   456.75 &    12.60 &    2
&    59.05 &     0.24 &    1
\\
OR + OtopRed
&  2402.87 &     3.62 &    1
&   467.93 &     0.04 &    1
&    60.06 &     0.02 &    1
\\
OtopRed + LexRed
&  2399.78 &     3.50 &    1
&   472.57 &     0.03 &    1
&    61.24 &     0.00 &    1
\\
OR + OtopRed + LexRed
&  2395.61 &     3.62 &    1
&   468.22 &     0.07 &    1
&    60.13 &     0.03 &    1
\\
\bottomrule
\end{tabular}%
\end{table}%
\begin{table}[b!]
\caption{MIPLIB results where the detected symmetry group is the same for the traditional setting and our generalized setting.}
\label{tab:miplib:same}
\centering
\footnotesize
\begin{tabular}{@{}L{3.2cm}*{3}{R{1.1cm}R{1cm}R{.6cm}}@{}}
\toprule
\multicolumn{1}{c}{Setting}
& \multicolumn{3}{c}{All instances (118)}
& \multicolumn{3}{c}{Solved by some setting (73)}
& \multicolumn{3}{c}{Solved by all settings (50)}
\\
\cmidrule(lr){1-1}
\cmidrule(lr){2-4}
\cmidrule(lr){5-7}
\cmidrule(lr){8-10}
& time (s) & sym (s) & \#S
& time (s) & sym (s) & \#S
& time (s) & sym (s) & \#S
\\
\cmidrule{2-10}
Nosym
&   899.10 &     0.18 &   58
&   381.92 &     0.17 &   58
&   160.65 &     0.08 &   50
\\
Polyh
&   677.77 &     1.67 &   66
&   241.61 &     1.06 &   66
&   140.79 &     0.71 &   50
\\
OF
&   750.98 &     1.23 &   65
&   285.34 &     0.97 &   65
&   136.50 &     0.49 &   50
\\
Polyh + OF
&   660.19 &     1.77 &   68
&   231.55 &     1.02 &   68
&   132.76 &     0.59 &   50
\\
\midrule
OtopRed
&   740.73 &     1.25 &   61
&   279.06 &     0.77 &   61
&   131.06 &     0.50 &   50
\\
LexRed
&   730.72 &     1.98 &   65
&   272.97 &     1.34 &   65
&   145.50 &     0.75 &   50
\\
OR
&   736.76 &     4.44 &   65
&   276.63 &     3.16 &   65
&   143.08 &     1.43 &   50
\\
OR + LexRed
&   717.69 &     4.88 &   66
&   265.14 &     3.29 &   66
&   140.54 &     1.61 &   50
\\
OR + OtopRed
&   656.95 &     2.29 &   65
&   229.74 &     1.40 &   65
&   130.59 &     0.59 &   50
\\
OtopRed + LexRed
&   621.98 &     1.79 &   67
&   210.24 &     1.04 &   67
&   125.43 &     0.57 &   50
\\
OR + OtopRed + LexRed
&   639.09 &     2.49 &   66
&   219.69 &     1.48 &   66
&   127.59 &     0.64 &   50
\\
\bottomrule
\end{tabular}%
\end{table}%

There are only 10 instances in the first class,
and only 2 of these can be solved by any of the settings considered.
So, indeed, 
as mentioned in the beginning of Section~\ref{sec:num},
the symmetrical instances from MIPLIB act predominantly on
binary variables.
Since only few instances handle symmetries based on a different group,
this only has a small effect on the reported results
in Section~\ref{sec:results:miplib}.
When restricting to the instances with the same symmetries
(Table~\ref{tab:miplib:same}),
the best of our methods on all instances
achieves an improvement of~\SI{5.8}{\percent}
over the traditional best, 
as opposed to the figure of~\SI{5.2}{\percent} from Section~\ref{sec:results:miplib}.
This shows that our generalized methods outperform similar
state-of-the-art methods even for problems where symmetries 
only act on binary variables.

\clearpage
\section{Testset generation for Noise Dosage instances}
\label{app:noisedosage}

In this appendix, we describe how we generate the instances of the noise
dosage problem used in our experiments.
We have extended the testset used by
Sherali and Smith~\cite{sherali2001models}, by including instances
with more workers and more machines.
Since the generator of thei instances is not available to us,
we have analyzed the instances from~\cite{sherali2001models}
provided by J. Cole Smith and extracted features.

Most importantly, we have observed that the total worker time is about 
half the total time required by the machines, and that the number of tasks
per machine range from 4 to 9. 
The noise dosage units and time per machine are floating point numbers.
Given $p$ machines and $q$ workers,
the workers work at most $H = 480$ hours,
we sample the number of tasks $d_i$ 
per machine uniformly (discrete) random between 4 and 10,
then choose $\mu = \frac{1}{2H} \sum_{i=1}^p d_i$
and sample the time per task from the normal distribution
with mean $\mu$ and standard deviation $\frac15 \mu$, 
i.e., $\mathcal N(\mu, \frac15 \mu)$.
Last, the noise units are sampled from $\mathcal N(18, 4)$.
When sampling from the normal distribution,
we ignore sample values with a negative value.

Our instance generator and generated instance files 
are publicly available 
at~\url{https://github.com/JasperNL/scip-unified/tree/unified/problem_instances}.

\end{document}